\documentclass[12pt]{amsart}

\usepackage[T1]{fontenc}
\usepackage[utf8]{inputenc}
\usepackage{amsmath,amsfonts,amssymb,amsxtra,setspace,xspace,graphicx,lmodern,fourier,mathrsfs}
\usepackage[colorlinks=true]{hyperref}
\hypersetup{urlcolor=blue, citecolor=red, linkcolor=blue}
\usepackage[usenames,dvipsnames]{color}

\newtheorem{theorem}{Theorem}
\newtheorem{lemma}[theorem]{Lemma}

\newtheorem{corollary}[theorem]{Corollary}
\newtheorem{proposition}[theorem]{Proposition}
\newtheorem{remark}{\sl Remark}
\newcommand{\Email}[1]{E-mail: \href{mailto:#1}{\textsf{#1}}}

\def\cprime{$'$}

\newcommand{\be}[1]{\begin{equation}\label{#1}}
\newcommand{\ee}{\end{equation}}
\renewcommand{\(}{\left(}
\renewcommand{\)}{\right)}
\newcommand{\R}{{\mathbb R}}

\newcommand{\irdx}[1]{\int_{\R^d}{#1}\,dx}
\newcommand{\irdxi}[1]{\int_\R{#1}\,dx}
\newcommand{\irdv}[1]{\int_{\R^d}{#1}\,dv}

\newcommand{\irdxv}[1]{\iint_{\R^d\times\R^d}#1\,dx\,dv}

\newcommand{\irdmu}[1]{\iint_{\R^d\times\R^d}#1\,d\mu}

\newcommand{\ir}[1]{\int_\R{#1}\,dx}

\newcommand{\nrm}[2]{\left\|{#1}\right\|_{#2}}

\newcommand{\scalar}[2]{\left\langle{#1},{#2}\right\rangle}
\newcommand{\nH}{\|}
\newcommand{\nHH}{\big\|^2}

\newcommand{\ve}{\varepsilon}
\newcommand{\veta}{\eta}

\newcommand{\ird}[1]{\int_{\R^d}{#1}\,dx}
\newcommand{\irh}[1]{\int_{\R^d}{#1}\,dv}

\newcommand{\sfA}{\mathsf A}
\newcommand{\sfD}{\mathsf D_\delta}
\newcommand{\sfH}{\mathsf H_\delta}
\newcommand{\sfL}{\mathsf L}
\newcommand{\sfPi}{\mathsf\Pi}
\newcommand{\sfT}{\mathsf T}
\newcommand{\Id}{\mathrm{Id}}

\renewcommand{\aa}{\mathsf s}
\newcommand{\irdM}[1]{\int_{\R^d}{#1}\,\mathcal M\,dv}

\newcounter{Hequation}

\makeatletter\g@addto@macro\equation{\stepcounter{Hequation}}\makeatother

\title[Linearized Vlasov-Poisson-Fokker-Planck system]{$\mathrm L^2$-Hypocoercivity and large time asymptotics of the linearized Vlasov-Poisson-Fokker-Planck system}

\author[L.~Addala, J.~Dolbeault, X.~Li, M.L.~Tayeb]{Lanoir Addala, Jean Dolbeault, Xingyu Li, and M. Lazhar Tayeb}

\date{\today -- File: \textsf{VPFP-ADLT-2021.tex}}
\begin{document}
\begin{abstract} This paper is devoted to the linearized Vlasov-Poisson-Fokker-Planck system in presence of an external potential of confinement. We investigate the large time behaviour of the solutions using hypocoercivity methods and a notion of scalar product adapted to the presence of a Poisson coupling. Our framework provides estimates which are uniform in the diffusion limit. As an application in a simple case, we study the one-dimensional case and prove the exponential convergence of the nonlinear Vlasov-Poisson-Fokker-Planck system without any small mass assumption. \end{abstract}
\keywords{Vlasov equation, Fokker-Planck operator, Poisson coupling, electrostatic forces, confinement, Vlasov-Poisson-Fokker-Planck system, convergence to equilibrium, large-time behavior, rate of convergence, hypocoercivity, diffusion limit, drift-diffusion equations.
}

\subjclass[2010]{Primary: 82C40. Secondary: 35H10; 35P15; 35Q84; 35R09; 47G20; 82C21; 82D10; 82D37.}
\maketitle
\thispagestyle{empty}

\section{Introduction and main results}\label{Sec:Intro}

The \emph{Vlasov-Poisson-Fokker-Planck system} in presence of an external potential $V$ is
\be{Syst:VPFP}\tag{VPFP}
\begin{array}{c}\displaystyle
\partial_tf+v\cdot\nabla_xf-\(\nabla_xV+\nabla_x\phi\)\cdot\nabla_vf=\Delta_vf+\nabla_v\cdot(v\,f)\,,\\[4pt]
\displaystyle-\Delta_x\phi=\rho_f=\irdv f\,.
\end{array}
\ee
In this paper, we shall assume that $(t,x,v)\in\R^+\times\R^d\times\R^d$ and that $\phi$ is a \emph{self-consistent potential} corresponding to repulsive electrostatic forces and that $V$ is a \emph{confining potential} in the sense that~\eqref{Syst:VPFP} admits, up to a multiplicative constant, a unique stationary solution 
\[
f_\star(x,v)=e^{-V(x)-\phi_\star(x)}\,\mathcal M(v)\,,\quad-\Delta_x\phi_\star=e^{-V-\phi_\star}\quad\mbox{and}\quad\mathcal M(v)=\frac{e^{-\frac12\,|v|^2}}{(2\pi)^{d/2}}\,,
\]
with associated potential $\phi_\star$ according, \emph{e.g.}, to~\cite{MR1994744,MR1677677}. We denote by $M=\irdxv{f_\star}$ the mass. System~\eqref{Syst:VPFP} is of interest for understanding the evolution of a system of charged particles with interactions of two different natures: a self-consistent, nonlinear interaction through the mean field potential~$\phi$ and collisions with a background inducing a diffusion and a friction represented by a Fokker-Planck operator acting on velocities. System~\eqref{Syst:VPFP} describes the dynamics of a plasma of Coulomb particles in a thermal reservoir (see for instance~\cite{MR961253}), but it has also been derived in stellar dynamics for gravitational models, as in~\cite{MR0031822}, in the case of an attractive mean field Newton-Poisson equation. Here we shall focus on the repulsive, electrostatic case. Applications range from plasma physics to semi-conductor modelling. A long standing open question is to get estimates on the rate of convergence to equilibrium in dimensions $d=2$ and $d=3$ for arbitrarily large initial data, away from equilibrium. We will not solve it here but, as an important step in this direction, we will establish a constructive estimate of the decay rate of the linearized problem, which provides us with an upper bound for the convergence rate of the nonlinear~\eqref{Syst:VPFP} problem. A technical but important issue is to decide how one should measure such a rate of relaxation. For this purpose, we introduce a norm which is adapted to the linearized problem and consistent with the \emph{diffusion limit}.

Let us consider the linearized problem around $f_\star$. Let $h$ be a function such that $f=f_\star\,(1+\veta\,h)$ with $\irdxv f=M$, that is, such that $\irdxv{h\,f_\star}=0$. The system~\eqref{Syst:VPFP} can be rewritten as
\[\hspace*{-5pt}
\begin{array}{c}\displaystyle
\partial_th+v\cdot\nabla_xh-\(\nabla_xV+\nabla_x\phi_\star\)\cdot\nabla_vh+v\cdot\nabla_x\psi_h-\Delta_vh+v\cdot\nabla_vh=\veta\(\nabla_x\psi_h\cdot\nabla_vh-v\cdot\nabla_x\psi_h h\)\\[4pt]
\displaystyle\mbox{with}\quad-\Delta_x\psi_h=\irdv{h\,f_\star}\,.
\end{array}
\]
At formal level, by dropping the $\mathcal O(\veta)$ term in the limit as $\eta\to0_+$, we obtain the \emph{linearized Vlasov-Poisson-Fokker-Planck system} around the equlibrium state $f_\star$ given by
\be{Syst:VPFPlin}
\begin{array}{c}\displaystyle
\partial_th+v\cdot\nabla_xh-\(\nabla_xV+\nabla_x\phi_\star\)\cdot\nabla_vh+v\cdot\nabla_x\psi_h-\Delta_vh+v\cdot\nabla_vh=0\,,\\[4pt]
\displaystyle-\Delta_x\psi_h=\irdv{h\,f_\star}\,,\quad\irdxv{h\,f_\star}=0\,.
\end{array}
\ee
{}From now on we shall say that $h$ has \emph{zero average} if $\irdxv{h\,f_\star}=0$. Let us define the norm
\be{Def:norm}
\nrm h{}^2:=\irdxv{h^2\,f_\star}+\irdx{|\nabla_x\psi_h|^2}\,.
\ee
Our main result is devoted to the large time behaviour of a solution of \emph{the linearized system}~\eqref{Syst:VPFPlin} on $\R^+\times\R^d\times\R^d\ni(t,x,v)$ with given initial datum $h_0$ at $t=0$. For simplicity, we shall state a result for a simple specific potential, but an extension to more general potentials will be given to the price of a rather long list of technical assumptions that are detailed in Section~\ref{Sec:FunctionalSetting}.
\begin{theorem}\label{Thm:Main} Let us assume that $d\ge1$, $V(x)=|x|^\alpha$ for some $\alpha>1$ and $M>0$. Then there exist two constants $\lambda>0$ and $\mathcal C>1$ such that any solution $h$ of~\eqref{Syst:VPFPlin} with an initial datum~$h_0$ of zero average with $\nrm{h_0}{}^2<\infty$ is such that
\be{Ineq:LinDecay}
\nrm{h(t,\cdot,\cdot)}{}^2\le\mathcal C\,\nrm{h_0}{}^2\,e^{-\lambda t}\quad\forall\,t\ge0\,.
\ee
\end{theorem}
The constant $\mathcal C$ in Theorem~\ref{Thm:Main} is larger than $1$ as a typical result of hypocoercivity methods. Indeed, since the Fokker-Planck operator acts only on the velocity variable $v$, an exponential decay with $\mathcal C=1$ cannot be expected for generic $x$-dependent functions. The main novelty here is that hypocoercive estimates can be obtained in presence of the non-local Poisson coupling in~\eqref{Syst:VPFPlin}, and not simply in some perturbative regime. The linearized problem~\eqref{Syst:VPFPlin} is at first sight easier than the full nonlinear system~\eqref{Syst:VPFP} but our result gives two crucial informations which are of importance for the linearized system as well as for the nonlinear one: 1) we prove an exponential decay rate, 2) we specify an appropriate functional space and a notion of distance, corresponding to the norm defined by~\eqref{Def:norm}, for measuring the convergence to equilibrium.

Our analysis is consistent with the \emph{diffusion limit} of the linearized system, as we shall explain below. For any $\ve>0$, we consider the solution of the \emph{linearized problem in the parabolic scaling} given by
\be{Syst:VPFPlinParab}
\begin{array}{c}\displaystyle
\ve\,\partial_th+v\cdot\nabla_xh-\(\nabla_xV+\nabla_x\phi_\star\)\cdot\nabla_vh+v\cdot\nabla_x\psi_h-\frac1\ve\,\big(\Delta_vh-v\cdot\nabla_vh\big)=0\,,\\[4pt]
\displaystyle-\Delta_x\psi_h=\irdv{h\,f_\star}\,,\quad\irdxv{h\,f_\star}=0\,.
\end{array}
\ee
In the regime $\ve\to0_+$, the parabolic scaling corresponds to a time scale of the order $\ve^{-1}$ with a collision frequency which is also of the order $\ve^{-1}$, but cannot be achieved by a simple scaling in the $(t,x,v)$ variables in presence of an external potential and a Stokes friction force. Our decay estimate is uniform with respect to $\ve\to0_+$.
\begin{theorem}\label{Thm:DiffusionLimit} Let us assume that $d\ge1$, $V(x)=|x|^\alpha$ for some $\alpha>1$ and $M>0$. For any $\ve>0$ small enough, there exist two constants $\lambda>0$ and $\mathcal C>1$, which do not depend on $\ve$, such that any solution $h$ of~\eqref{Syst:VPFPlinParab} with an initial datum~$h_0$ of zero average and such that $\nrm{h_0}{}^2<\infty$ satisfies~\eqref{Ineq:LinDecay}.\end{theorem}
The result of Theorem~\ref{Thm:Main} will be extended in Theorem~\ref{Thm:MainBis} to a larger class of external potentials~$V$: in the technical part of the proof of Theorem~\ref{Thm:Main}, we will specify precise but more general conditions under which the same result holds. A similar extension applies in the case of Theorem~\ref{Thm:DiffusionLimit}. As an application of our method, we establish the exponential rate of convergence of the solution of \emph{the non-linear system}~\eqref{Syst:VPFP} when $d=1$. For sake of simplicity, we state the result for the same potential $V$ as in Theorem~\ref{Thm:DiffusionLimit}.
\begin{corollary}\label{Cor:d=1} Assume that $d=1$, $V(x)=|x|^\alpha$ for some $\alpha>1$ and $M>0$. If $f$ solves~\eqref{Syst:VPFP} with initial datum~$f_0=(1+h_0)\,f_\star$ such that $h_0$ has zero average, $\nrm{h_0}{}^2<\infty$ and $(1+h_0)\ge0$, then~\eqref{Ineq:LinDecay} holds with $h=f/f_\star-1$ for some constants~$\lambda>0$ and $\mathcal C>1$.\end{corollary}

\medskip The \emph{diffusion limit} of systems of kinetic equations in presence of electrostatic forces has been studied in many papers. The mathematical results go back at least to the study of a model for semi-conductors involving a linear Boltzmann kernel by F.~Poupaud in~\cite{MR1127004}. The case of a Fokker-Planck operator in dimension $d=2$ was later studied by F.~Poupaud and J.~Soler in~\cite{MR1780148}, and by T.~Goudon in~\cite{MR2139941}, on the basis of the existence results of~\cite{MR771811,MR1052014}. There is also a parallel, probabilistic approach of the macroscopic diffusion limit and of the overdamped regime for the generalized Langevin equation: see~\cite{Iacobucci_2017,Pavliotis_2021} and references therein. With a self-consistent Poisson coupling, we refer to~\cite{MR1200643} for existence results in dimension $d=3$ and to~\cite{MR1058151,MR1677677} for steady states, confinement and related issues. Based on free energy considerations introduced in~\cite{MR1306570,MR1677677}, N.~El Ghani and N.~Masmoudi were able in~\cite{MR2664460} to establish diffusion limits also when $d=3$. Altogether, it is proved in dimensions $d=2$ and $d=3$ that the Vlasov-Poisson-Fokker-Planck system, with parameters corresponding to the parabolic scaling,
\be{VPFPeps}
\ve\,\partial_tf+v\cdot\nabla_xf-\(\nabla_xV+\nabla_x\phi\)\cdot\nabla_vf=\frac1\ve\(\Delta_vf+\nabla_v\cdot(v\,f)\)\,,\quad-\Delta_x\phi=\rho_f=\irdv f\,,
\ee
has a weak solution $\big(f^\ve,\phi^\ve\big)$ which converges as $\ve\to0_+$ to $\big(f^0=\rho\,\mathcal M,\phi\big)$ where $\mathcal M(v)=(2\pi)^{-d/2}\,\exp(-|v|^2/2)$ is the normalized Maxwellian function and where the charge density $\rho=\irdv{f^0}$ is a weak solution of the \emph{drift-diffusion-Poisson} system
\be{ddP}
\frac{\partial\rho}{\partial t}=\nabla_x\cdot\big(\nabla_x\rho+\rho\,\nabla_x(V+\phi)\big)\,,\quad-\Delta_x\phi=\rho\,.
\ee
Another piece of information is the asymptotic behavior of the solutions of~\eqref{ddP} for large times. As $t\to+\infty$, it is well known (see for instance~\cite{MR2110936} in the case of a bounded domain,~\cite{AMTU} in the Euclidean case when $V(x)=|x|^2$, and~\cite{MR1777308} in $\R^d$ with a confining external potential $V$ for any $d\ge3$) that the solution of~\eqref{ddP} converges to a steady state $(\rho_\star,\phi_\star)$ given by
\be{Eqn:Poisson-Boltzmann}
-\Delta_x\phi_\star=\rho_\star=e^{-V-\phi_\star}
\ee
at an exponential rate. The optimal asymptotic rates have been characterized recently in~\cite{Nernst-Planck} using the linearized drift-diffusion-Poisson system and a norm which involves the Poisson potential. Apart the difficulty arising from the self-consistent potential, the technique is based on relative entropy methods, which are by now standard in the study of large time asymptotics of drift-diffusion equations.

Our goal is to study both regimes $\ve\to0_+$ and $t\to+\infty$ simultaneously. More precisely, we aim at proving that each solution $\big(f^\ve,\phi^\ve\big)$ of~\eqref{VPFPeps} converges to $(f_\star,\phi_\star)$ as $t\to+\infty$ in a weighted $\mathrm L^2$ sense at an exponential rate which is uniform in $\ve>0$, small. In the present paper, we will focus on a linearized regime in any dimension and obtain an estimate of the decay rate in the asymptotic regime. This allows us to obtain an asymptotic decay rates in the non-linear regime when $d=1$, but so far not in higher dimensions because we are still lacking some key estimates. Compared to the large time asymptotics of~\eqref{ddP}, the study of the convergence rate of the solution of~\eqref{VPFPeps} or, in the case $\ve=1$, of the decay rate of the solution of~\eqref{Syst:VPFPlin}, is much more difficult because the diffusion only acts on the velocities and requires the use of hypocoercive methods.

\medskip T.~Gallay coined the word \emph{hypocoercivity} in the context of convergence without regularization as opposed to hypoellipticity where both properties arise simultaneously. This concept is well adapted to kinetic equations with general collision kernels and C.~Villani made the \emph{hypocoercivity} very popular in kinetic theory: see~\cite{MR2275692,MR2562709}. Understanding the large time behavior of the kinetic Fokker-Planck equation (without Poisson coupling) is an interesting problem which has a long history: see~\cite{MR1503147,MR0222474,MR0168018,MR1969727,MR2034753} for some earlier contributions. C.~Villani~\cite{MR2562709} proved convergence results in various senses: in $\mathrm H^1$~\cite[Theorem~35]{MR2562709}, in $\mathrm L^2$~\cite[Theorem~37]{MR2562709}, and in entropy~\cite[Theorem~39]{MR2562709} when $\mathrm{Hess}(V)$ is bounded. His approach is however inspired by hypoelliptic methods, as in~\cite{Herau,MR2294477,mouhot2006quantitative}. The method of~\cite{arnold2014sharp} is based on a spectral decomposition and produces an exponential decay in relative entropy with a sharp rate. In a somewhat similar spirit, we can also quote~\cite{BDMMS}, which is based on a Fourier decomposition. Due to the Fokker-Planck operator, smoothing effects in~\eqref{VPFPeps} can be expected as was proved in~\cite{MR1355890}, consistently with hypoelliptic methods: this will not be exploited in the present paper.

Mean-field couplings add a serious difficulty: see~\cite{guillin:hal-02387517,Lu_2019} for recent results based on a probabilistic approach. In presence of a Poisson coupling large time behavior (without rates) of the solutions of~\eqref{VPFPeps} has been dealt with in presence of or without an external potential: \emph{cf.}~\cite{MR1306570,MR1414375,MR1677677,carpio1998long,Kagei_2001} for early results. In~\cite{Ju_Hwang_2013}, a result of exponential decay is obtained in dimension $d=3$, in presence of a constant neutralizing background but without confinement: the solution is a smooth perturbation of a stationary distribution function which is homogeneous in $x$ and Maxwellian in $v$ and the proof relies on remarkable algebraic properties. When $d=2$ and $d=3$, F.~H\'erau and L.~Thomann~\cite{MR3522010} proved the trend to the equilibrium for the Vlasov-Poisson-Fokker-Planck system with a small nonlinear term but with a possibly large exterior confining potential. More recently, M.~Herda and M.~Rodrigues considered in~\cite{MR3767000} the two limits as $\ve\to0_+$ and $t\to+\infty$, on the $2$-dimensional torus, in the globally neutral regime. By a careful analysis of the trade-off between two parameters, the mean free path and the Debye length, they establish closed estimates of regularity which allow them to prove an exponential convergence, including in various limiting regimes, with uniform estimates in the other, fixed parameters. All these approaches are essentially of perturbative nature. In various papers, the properly linearized system~\eqref{Syst:VPFPlinParab} is not taken into account, in the sense that the non-local term arising from the Poisson equation is often dropped. In the case of a torus and without an external potential, the Landau damping provides another mechanism of convergence to equilibrium even without a Fokker-Planck kernel: we refer to~\cite{Bedrossian_2017} for a detailed study by J.~Bedrossian on the enhancement induced by the Fokker-Planck operator acting on velocities and also to a result of I.~Tristani in~\cite{Tristani_2017} for the analysis of the consequences of the Landau damping on the (properly) linearized Vlasov–Poisson-Fokker-Planck system. So far it is not known how these properties could be extended from the setting of a torus to the case of the whole Euclidean space in presence of an external potential of confinement. Let us emphasize that, in the present paper, we consider the properly linearized system, including the non-local Poisson term, and provide a functional framework which is compatible with hypocoercivity methods adapted to diffusion limits.

The existence of solutions of~\eqref{Syst:VPFPlin}, which are continous w.r.t.~$t$ and take values in $\mathrm L^2$ for the norm defined by~\eqref{Def:norm}, is out of the scope of this paper. Seen as a perturbation of~\eqref{Syst:VPFP}, an existence result can be deduced from the results of~\cite{MR1052014,MR1200643} or established directly using the same methods as in these papers and we will consider it as granted. Alternatively, it is also possible to consider the non-local term as perturbation and use a fixed point argument based on the semi-group associated to the Fokker–Planck operator as, \emph{e.g.}, in~\cite{MR3522010}.

In~\cite{MR3324910}, J.~Dolbeault, C.~Mouhot, and C.~Schmeiser studied the exponential decay in a modified $\mathrm L^2$ norm for the Vlasov-Fokker-Planck equation (and also for a larger class of linear kinetic equations). The method was motivated by the results of~\cite{Herau} but the main source of inspiration came from the analysis of the diffusion limit, as in~\cite{MR2082927,MR2299429,MR2338354} (also see~\cite{MR3535871} in presence of an oscillating external force field): the general idea is to build a norm which reflects the spectral gap that determines the rate of convergence in~\eqref{ddP} by adding a twist which arises from the coercivity properties, at macroscopic level, of the diffusion limit. Applying~\cite{MR3324910} to~\eqref{Syst:VPFPlin} is a natural idea, which is mentioned for instance in~\cite[p.~109]{Tristani_2017}, but has not been done yet to our knowledge. Inspired by~\cite{MR2481073,MR2726546,1004}, another idea emerged that asymptotic rates of convergence should be measured in a norm induced by a Taylor expansion of the entropy around the asymptotic state and that, in presence of a Poisson coupling, this norm should involve a non-local term: see~\cite{MR3196188,Nernst-Planck,li:hal-02143985} for drift-diffusion systems and Remark~\ref{Rem:FreeEnergy} when applied to~\eqref{Syst:VPFP}. The goal of this paper is to mix these two ideas. It turns out that they combine into a beautiful machinery.

\medskip This paper is organized as follows. In Section~\ref{Sec:Abstract}, we expose the strategy for the $\mathrm L^2$-hypocoercivity method of~\cite{MR3324910} in the abstract setting of a general Hilbert space. The notion of Hilbert space adapted to~\eqref{Syst:VPFPlin} is exposed in Section~\ref{Sec:FunctionalSetting} with some fundamental considerations on confinement by an external potential and adapted Poincar\'e inequalities. Section~\ref{Sec:Proof} is devoted to the proof of Theorem~\ref{Thm:Main}: we have to check that the assumptions of Section~\ref{Sec:Abstract} hold in the functional setting of Section~\ref{Sec:FunctionalSetting}, with the special scalar product for Poisson coupling involving a non-local term associated with the norm defined by~\eqref{Def:norm}. In Section~\ref{Sec:DiffusionLimit}, we prove Theorem~\ref{Thm:DiffusionLimit}: our estimates are compatible with the diffusion limit as $\ve\to0$. Coming back to the non-linear problem~\eqref{Syst:VPFP} in dimension $d=1$, we prove in this latter case that an exponential rate of convergence as $t\to+\infty$ can be measured in the hypocoercive norm, that is, we prove Corollary~\ref{Cor:d=1}.

\medskip We shall adopt the following conventions. If $\mathsf a=(\mathsf a_i)_{i=1}^d$ and $\mathsf b=(\mathsf b_i)_{i=1}^d$ are two vectors with values in $\R^d$, then $\mathsf a\cdot\mathsf b=\sum_{i=1}^d\mathsf a_i\,\mathsf b_i$ and $|\mathsf a|^2=\mathsf a\cdot\mathsf a$. If $\mathsf A=(\mathsf A_{ij})_{i,j=1}^d$ and $\mathsf B=(\mathsf B_{ij})_{i,j=1}^d$ are two matrices with values in $\R^d\times\R^d$, then \hbox{$\mathsf A:\mathsf B=\sum_{i,j=1}^d\mathsf A_{ij}\,\mathsf B_{ij}$} and $|\mathsf A|^2=\mathsf A:\mathsf A$. We shall use the tensor convention that $\mathsf a\otimes\mathsf b$ is the matrix of elements $\mathsf a_i\,\mathsf b_j$. By extension to functions, $\nabla_xw$ is the gradient of a scalar function $w$ while $\nabla_x\cdot\mathsf u$ denotes the divergence of a vector valued function $\mathsf u=(\mathsf u_i)_{i=1}^d$ and $\nabla_x\otimes\mathsf u$ is the matrix valued function of elements $\partial\mathsf u_i/\partial x_j$. Hence
\[\textstyle
\mathrm{Hess}(w)=(\nabla_x\otimes\nabla_x)w=\(\frac{\partial^2w}{\partial x_i\,\partial x_j}\)_{i,j=1}^d
\]
denotes the Hessian of $w$ and, for instance, $\mathsf u\otimes \mathsf u:\mathrm{Hess}(w)=\sum_{i,j=1}^d\mathsf u_i\,\mathsf u_j\,\big(\mathrm{Hess}(w)\big)_{ij}$. We shall also write that $|\mathrm{Hess}(w)|^2=\mathrm{Hess}(w):\mathrm{Hess}(w)$.

\section{Hypocoercivity result and decay rates}\label{Sec:Abstract}

This section is devoted to the abstract hypocoercivity method in general Hilbert spaces and it is inspired from~\cite{MR3324910,BDMMS}. Since the methods sets the overall strategy of proof of our main results, we expose it for the convenience of the reader.

\medskip Let us consider the evolution equation
\be{EqnEvol}
\frac{dF}{dt}+\sfT F=\sfL F
\ee
on a Hilbert space $\mathcal H$. In view of the applications, we shall call $\sfT$ and $\sfL$ the \emph{transport} and the \emph{collision} operators and assume without further notice that they are respectively antisymmetric and symmetric, and both time-independent. On $\mathcal H$, we shall denote by~$\scalar\cdot\cdot$ and $\|\cdot\nH$ the scalar product and the norm. As in~\cite{MR3324910}, we assume that there are positive constants $\lambda_m$, $\lambda_M$, and $C_M$ such that, for any $F\in\mathcal H$, the following properties hold:\\[4pt]
$\rhd$ \emph{microscopic coercivity}
\be{H1}\tag{H1}
-\,\scalar{\sfL F}F\ge\lambda_m\,\|(\Id-\sfPi)F\nH^2\,,
\ee
$\rhd$ \emph{macroscopic coercivity}
\be{H2}\tag{H2}
\|\sfT\sfPi F\nH^2\ge\lambda_M\,\|\sfPi F\nH^2\,,
\ee$\rhd$ \emph{parabolic macroscopic dynamics}
\be{H3}\tag{H3}
\sfPi\sfT\sfPi\,F=0\,,
\ee$\rhd$ \emph{bounded auxiliary operators}
\be{H4}\tag{H4}
\|\sfA\sfT(\Id-\sfPi)F\nH+\|\sfA\sfL F\nH\le C_M\,\|(\Id-\sfPi)F\nH\,.
\ee
Here $\Id$ is the identity, $\sfPi$ is the orthogonal projection onto the null space of $\sfL$, ${}^*$ denotes the adjoint with respect to $\scalar\cdot\cdot$ and as in~\cite{Dolbeault2009511,MR3324910}, the operator $\sfA$ is defined by
\[
\sfA:=\big(\Id+(\sfT\sfPi)^*\sfT\sfPi\big)^{-1}(\sfT\sfPi)^*.
\]

Since a solution $F$ of~\eqref{EqnEvol} obeys to
\[
\frac12\,\frac d{dt}\|F\nH^2=\scalar{\sfL F}F\le-\,\lambda_m\,\|(\Id-\sfPi)F\nH^2\,,
\]
this is not enough to conclude that $\|F(t,\cdot)\nH^2$ decays exponentially with respect to $t\ge0$ and this is why we shall consider the Lyapunov functional
\[
\sfH[F]:=\tfrac12\,\|F\nH^2+\delta\,\scalar{\sfA F}F
\]
for some $\delta>0$ to be determined later. If $F$ solves~\eqref{EqnEvol}, then
\[
-\,\frac d{dt}\sfH[F]=\sfD[F]:=-\,\scalar{\sfL F}F+\delta\,\scalar{\sfA\sfT\sfPi F}F-\,\delta\,\scalar{\sfT\sfA F}F+\delta\,\scalar{\sfA\sfT(\Id-\sfPi)F}F-\delta\,\scalar{\sfA\sfL F}F
\]
using $\mathsf A = \sfPi\mathsf A$. Let us define
\[
\delta_\star=\min\left\{2,\,\lambda_m,\,\frac{4\,\lambda_m\,\lambda_M}{4\,\lambda_M+C_M^2\,(1+\lambda_M)}\right\}\,.
\]
We recall that the two main properties of the \emph{hypocoercivity} method of~\cite{MR3324910} for real valued operators and later extended in~\cite{BDMMS} to complex Hilbert spaces go as follows.
\begin{proposition}\label{DMS} Assume that~\eqref{H1}--\eqref{H4} hold and take $\delta\in(0,\delta_\star)$. Then we have:\\
{\rm (i)} $\sfH$ and $\|\cdot\nH^2$ are equivalent in the sense that
\be{H-norm}
\frac{2-\,\delta}4\,\|F\nH^2\le\sfH[F]\le\frac{2+\,\delta}4\,\|F\nH^2\quad\forall F\in\mathcal H\,.
\ee
{\rm (ii)} For some $\lambda>0$ depending on $\delta$, $\sfH$ and $\sfD$ are related by the entropy -- entropy production inequality
\be{Decay:H}
\lambda\,\sfH[F]\le\sfD[F]\quad\forall F\in\mathcal H\,.
\ee
\end{proposition}
\noindent As a consequence, a solution $F$ of~\eqref{EqnEvol} with initial datum $F_0$ obeys to
\[
\sfH[F(t,\cdot)]\le\sfH[F_0]\,e^{-\lambda t}
\]
and
\be{HypocoRate}
\|F(t,\cdot)\nH^2\le\frac4{2-\delta}\,\sfH[F(t,\cdot)]\le\frac4{2-\delta}\,\sfH[F_0]\,e^{-\lambda t}\le\frac{2+\delta}{2-\delta}\,\|F_0\nH^2\,e^{-\lambda t}\quad\forall\,t\ge0\,.
\ee
\begin{proof} For completeness, we sketch the main steps of the proof, with slightly improved estimates compared to~\cite[Theorem~3]{BDMMS}. Since $\sfA\sfT\sfPi$ can be viewed as $z\mapsto(1+z)^{-1}\,z$ applied to $(\sfT\sfPi)^*\sfT\sfPi$,~\eqref{H1} and~\eqref{H2} imply~that
\[
-\,\scalar{\sfL F}F+\delta\,\scalar{\sfA\sfT\sfPi F}F\ge\lambda_m\,\|(\Id-\sfPi)F\nH^2+\frac{\delta\,\lambda_M}{1+\lambda_M}\,\|\sfPi F\nH^2\,.
\]
Our goal is to prove that the r.h.s.~controls the other terms in the expression of $\sfD[F]$. By~\eqref{H4}, we know that
\[
\big|\scalar{\sfA\sfT(\Id-\sfPi)F}F+\,\scalar{\sfA\sfL F}F\big|\le C_M\,\|\sfPi F\nH\,\|(\Id-\sfPi)F\nH\,.
\]
As in~\cite[Lemma~1]{MR3324910}, if $G=\sfA F$, \emph{i.e.}, if $(\sfT\sfPi)^*F=G+(\sfT\sfPi)^*\sfT\sfPi\,G$, then
\[
\scalar{\sfT\sfA F}F=\scalar G{(\sfT\sfPi)^*F}=\|G\nH^2+\|\sfT\sfPi G\nH^2=\|\sfA F\nH^2+\|\sfT\sfA F\nH^2\,.
\]
By the Cauchy-Schwarz inequality, we know that
\[
\scalar G{(\sfT\sfPi)^*F}=\scalar{\sfT\sfA F}{(\Id-\sfPi)F}\\
\le\|\sfT\sfA F\nH\,\|(\Id-\sfPi)F\nH\le \tfrac1{2\,\mu}\,\|\sfT\sfA F\nH^2+ \tfrac\mu2\,\|(\Id-\sfPi)F\nH^2
\]
for any $\mu>0$. Hence
\[\label{A-bound}
2\,\|\sfA F\nH^2+\big(2-\tfrac1\mu\big)\,\|\sfT\sfA F\nH^2\le\mu\,\|{(\Id-\sfPi)}F\nH^2\,,
\]
which, by taking either $\mu=1/2$ or $\mu=1$, proves that
\be{Eqn:AF-TAF}
\|\sfA F\nH\le \tfrac12\,\|{(\Id-\sfPi)}F\nH\quad\mbox{and}\quad\|\sfT\sfA F\nH\le\|{(\Id-\sfPi)}F\nH\,.
\ee
Hence
\[
|\scalar{\sfA F}F|\le\tfrac12\,X\,Y\le\tfrac14\(X^2+Y^2\)
\]
with $X:=\|(\Id-\sfPi)F\nH$ and $Y:=\|\sfPi F\nH$ because $\sfA$ takes values in $\sfPi\mathcal H$. This establishes~\eqref{H-norm} and, as a side result, also proves that
\[
\left|\scalar{\sfT\sfA F}F\right|=\left|\scalar{\sfT\sfA F}{(\Id-\sfPi)F}\right|\le\|{(\Id-\sfPi)}F\nH^2\,.
\]
Collecting terms in the expression of $\sfD[F]$, we find that
\[
\sfD[F]\ge(\lambda_m-\,\delta)\,X^2+\frac{\delta\,\lambda_M}{1+\lambda_M}\,Y^2-\,\delta\,C_M\,X\,Y
\]
with $X:=\|(\Id-\sfPi)F\nH$ and $Y:=\|\sfPi F\nH$. We know that $\sfH[F]\le\frac12\(X^2+Y^2\)+\frac\delta2\,X\,Y$, so that the largest value of $\lambda$ for which $\sfD[F]\ge\lambda\,\sfH[F]$ can be estimated by the largest value of $\lambda$ for which
\[
(X,Y)\mapsto(\lambda_m-\,\delta)\,X^2+\frac{\delta\,\lambda_M}{1+\lambda_M}\,Y^2-\,\delta\,C_M\,X\,Y-\frac\lambda2\,\(X^2+Y^2\)-\frac\lambda2\,\delta\,X\,Y=(X,Y)M(X,Y)^T
\]
is a nonnegative quadratic form, as a function of $(X,Y)$, where 
\[
M:=\(\begin{array}{cc}
\lambda_m-\,\delta-\frac\lambda2&-\frac\delta2\(C_M+\frac\lambda2\) \\[8pt]
-\frac\delta2\(C_M+\frac\lambda2\)&\frac{\delta\,\lambda_M}{1+\lambda_M}-\frac\lambda2
\end{array}\)\,.
\]
It is characterized by the discriminant condition
\[
h(\delta,\lambda):=\delta^2\(C_M+\frac\lambda2\)^2-4\,\(\lambda_m-\,\delta-\frac\lambda2\)\(\frac{\delta\,\lambda_M}{1+\lambda_M}-\frac\lambda2\)\le0
\]
and the sign condition~$\lambda_m-\,\delta-\lambda/2>0$. For any $\delta\in(0,\delta_\star)$, $\lambda\mapsto h(\delta,\lambda)$ is a second order polynomial such that $h(\delta,0)>0$ and $\lim_{\lambda\to+\infty}h(\delta,\lambda)=-\infty$. Hence, the largest possible value of $\lambda$ can be estimated by the positive root of $h(\delta,\lambda)=0$, for any given $\delta\in(0,\delta_\star)$.\end{proof}

Notice that the proof of Proposition~\ref{DMS} provides us with a constructive estimate of the decay rate $\lambda$, as a function of $\delta\in(0,\delta_\star)$. We refer to~\cite{ADSW} for a discussion of the best estimate of the decay rate of $\sfH$, \emph{i.e.}, the largest possible estimate of $\lambda$ when $\delta$ varies in the admissible range $(0,\delta_\star)$.

\section{Functional setting}\label{Sec:FunctionalSetting}

In this section, we collect a some observations on the external potential $V$ and on the stationary solution obtained by solving the Poisson-Boltzmann equation. Depending on growth conditions on $V$, we establish a notion of \emph{confinement} (so that~\eqref{Syst:VPFP} admits an integrable stationary solution) and coercivity properties (which amount to Poin\-car\'e type inequalities). Our goal is to give sufficient conditions in order that:\\
1) there exists a nonnegative stationary solution $f_\star$ of~\eqref{Syst:VPFP} of arbitrary given mass $M>0$: see Section~\ref{Sec:Poisson-Boltzmann};\\
2) there is a Poincar\'e inequality associated with the measure $e^{-V-\phi_\star}dx$ on $\R^d$, and variants of it, with weights: see Section~\ref{Sec:Poincare};\\
3) there is a Hilbert space structure on which we can study~\eqref{Syst:VPFPlin}: see Section~\ref{Sec:ScalarProduct}.\\
These conditions on $V$ determine a functional setting which is adapted to implement the method of Section~\ref{Sec:Abstract}. The potential $V(x)=|x|^\alpha$ with $\alpha>1$ is an \emph{admissible potential} in that perspective.

In~\cite{MR3324910}, without Poisson coupling, sufficient conditions were given on $V$ which were inspired by the \emph{carr\'e du champ} method and the Holley-Stroock perturbation lemma (see~\cite{MR893137} and~\cite{doi:10.1142/S0218202518500574} for related results). These conditions are not well adapted to handle an additional Poisson coupling. Here we adopt a slightly different approach, which amounts to focus on sufficient growth conditions of the external potential $V$ and on tools of spectral theory like Persson's lemma. For sake of simplicity, we require some basic regularity properties of $V$ and assume that
\be{Hyp-V1}\tag{V1}
V\in C^0\cap\mathrm W_{\mathrm{loc}}^{2,1}\big(\R^d\big)\quad\mbox{and}\quad\liminf_{|x|\to+\infty}V(x)=+\infty\,.
\ee
These regularity assumptions and the growth conditions on $V$ (also see below) could be relaxed, up to additional technicalities.

\subsection{Preliminary considerations on the Poisson equation and conventions}\label{Sec:Prelim}

Let us consider the Green function $G_d$ associated with $-\Delta_x$. We shall write $\phi=(-\Delta_x)^{-1}\,\rho$ as a generic notation for $\phi=G_d\ast\rho$ with $G_d(x)=c_d\,|x|^{2-d}$, $c_d^{-1}=(d-2)\,|\mathbb S^{d-1}|$ if $d\ge3$. Then, if $d\ge3$, with no further restriction, by using integrations by parts, we have that
\[
\irdx{\rho\,\phi}=\irdx{(-\Delta_x\phi)\,\phi}=\irdx{|\nabla_x\phi|^2}\,.
\]
If $d=2$, we use $G_2(x)=-\frac1{2\pi}\,\log{|x|}$. It is a standard observation that $\phi=(-\Delta_x)^{-1}\,\rho$ is such that $\nabla_x\phi(x)=-\frac1{2\pi}\,\big(\int_{\R^2}\rho\,dx\big)\,\frac x{|x|^2}$ as $|x|\to+\infty$ is not square integrable unless $\int_{\R^2}\rho\,dx=0$. If $\int_{\R^2}\rho\,dx=0$, one can prove that
\[
\int_{\R^2}\rho\,\phi\,dx=\int_{\R^2}|\nabla_x\phi|^2\,dx<+\infty\,.
\]

If $d=1$, we can use $G_1(x)=-\,|x|/2$, but it is sometimes more convenient to rely on the representation
\be{Rep1d}
\phi(x)=\frac M2\,x-\int_{-\infty}^xdy\int_{-\infty}^y\rho(z)\,dz+\phi_0
\ee
and we shall then write $\phi=\big(-d^2/dx^2\big)^{-1}\rho$ whenever we use~\eqref{Rep1d}. Here $\phi_0$ is a free parameter that we may fix by assuming that $\max_{x\in\R}\phi(x)=0$. This convention does not coincide with the representation by $G_1*\rho$ as, in that case, $\phi(0)$ is taken equal to $-\frac12\ir{|x|\,\rho(x)}$, provided this quantity is well defined. Obviously, the potential is defined up to an additive constant and this is therefore not an issue. Without further notice, we will rely on~\eqref{Rep1d} for the solutions when $d=1$.

If $d=1$, more important is the fact that $\phi=\(-d^2/dx^2\)^{-1}\rho$ satisfies $\phi'=-\,\mathsf m$ where $\mathsf m$ is defined by $\mathsf m(x):=\int_{-\infty}^x\rho(y)\,dy$ if $M=\int_\R\rho\,dx=0$. In that case, if we further assume that $\rho$ is compactly supported or has a sufficient decay at infinity, an integration by parts shows~that
\be{Nonnegativity1d}
\int_\R\phi\,\rho\,dx=-\int_\R\phi'\,\mathsf m\,dx=\int_\R|\phi'|^2\,dx=\int_\R\mathsf m^2\,dx\ge0\,.
\ee

Altogether, whenever $\irdx\rho=0$, we shall write $\irdx{\rho\,\phi}=\irdx{|\nabla_x\phi|^2}\ge0$ without any further precaution, for any $d\ge1$.

\subsection{The Poisson-Boltzmann equation}\label{Sec:Poisson-Boltzmann}

According to~\cite{MR1058151,MR1052014,MR1677677}, \emph{stationary solutions} of the \eqref{Syst:VPFP} system are given by
\[
f_\star(x,v)=\rho_\star(x)\,\mathcal M(v)
\]
where $\mathcal M(v)=(2\pi)^{-d/2}\,e^{-|v|^2/2}$ is the normalized Maxwellian function (or Gaussian function) and the spatial density $\rho_\star$ is determined by the \emph{Poisson-Boltzmann equation}
\[
-\Delta_x\phi_\star=\rho_\star=M\,\frac{e^{-V-\phi_\star}}{\irdx{e^{-V-\phi_\star}}}\,.
\]
This equation also appears in the literature as the the \emph{Poisson-Emden equation}. It is obvious that $\phi_\star$ is defined up to an additive constant which can be chosen such that $M=\irdx{e^{-V-\phi_\star}}$ and therefore solves~\eqref{Eqn:Poisson-Boltzmann}. Here $\nrm{\rho_\star}{\mathrm L^1(\R^d)}=\nrm{f_\star}{\mathrm L^1(\R^d\times\R^d)}=M$ is the mass, which is a free parameter of the problem, which can be adjusted by adding a constant to~$V$. The critical growth of $V$ such that there are solutions $\rho_\star\in\mathrm L^1(\R^d)$ of~\eqref{Eqn:Poisson-Boltzmann} which minimize the free energy strongly depends on the dimension. Here are some sufficient conditions.
\begin{lemma}\label{Lem:Poisson-Boltzmann} Let $M>0$. Assume that $V$ satisfies~\eqref{Hyp-V1} and
\be{Hyp-V2}\tag{V2}
\begin{array}{lll}
&|V|\,e^{-V}\in\mathrm L^1(\R^d)\quad&\mbox{if}\quad d\ge3\,,\\[4pt]
&\liminf_{|x|\to+\infty}\frac{V(x)}{\log|x|}>2+\frac M{2\pi}\quad&\mbox{if}\quad d=2\,,\\[4pt]
&\liminf_{|x|\to+\infty}\frac{V(x)-M\,|x|/2}{\log|x|}>2\quad&\mbox{if}\quad d=1\,.
\end{array}
\ee
Then~\eqref{Eqn:Poisson-Boltzmann} has a unique solution $\rho_\star\in\mathrm L^1(\R^d)$ such that $\irdx{\rho_\star}=M$ and $\phi_\star$ is the unique solution of~\eqref{Eqn:Poisson-Boltzmann}. Moreover~$\phi_\star$ is of class $C^2$ and $\liminf_{|x|\to+\infty}W_\star(x)=+\infty$, where
\[
W_\star=V+\phi_\star\quad\mbox{and}\quad\rho_\star=e^{-W_\star}\,.
\]
\end{lemma}
As a consequence of Lemma~\ref{Lem:Poisson-Boltzmann}, we learn that under Assumptions~\eqref{Hyp-V1} and~\eqref{Hyp-V2}, the potential $W_\star$ also satisfies~\eqref{Hyp-V1}. Regularity results on~\eqref{Eqn:Poisson-Boltzmann} are scattered in the literature. See for instance~\cite[Proposition~3.5]{MR3522010}. The general strategy is, as usual, to use the fact that the solution is in the energy space and the equation to obtain uniform estimates by elliptic bootstrapping. The regularity and decay estimates as $|x|\to+\infty$ follow respectively from the regularity of $V$ and from its growth properties, using a representation of the solution based on the Green function. This is again classical and details will be omitted here.
\begin{proof} The case $d\ge3$ is covered by~\cite[p.~123]{MR1677677}. The free energy
\[
\mathcal F[\rho]:=\irdx{\rho\,\log\rho}+\irdx{\rho\,V}+\frac12\irdx{\rho\,\phi}
\]
is bounded from below under the mass constraint $\irdx\rho=M$ using the fact that 
\[
\irdx{\rho\,\phi}=\irdx{|\nabla_x\phi|^2}\ge0
\]
and, with $u:=\rho\,e^V$, the convexity estimate
\[
\mathcal F[\rho]\ge\irdx{\rho\,\log\rho}+\irdx{\rho\,V}=\irdx{(\underbrace{u\,\log u-u+1}_{\ge0})\,e^{-V}}+M-\ird{e^{-V}}\,.
\]
A slightly more accurate estimate is obtained using Jensen's inequality applied to $u\,\log u$ with the measure $e^{-V}\,dx/\ird{e^{-V}}$. The existence follows by a minimization method. The Euler-Lagrange equation reads
\[
1+\log\rho+V+\phi=C
\]
for some $C\in\R$ which is the Lagrange multiplier associated to the mass constraint. Hence $\rho=e^{C-1}\,e^{-(V+\phi)}$ and we deduce from the mass constraint $M=\ird\rho$ that $C=1+\log M-\log\ird{e^{-(V+\phi)}}$. As noticed in~\cite{MR1015923,MR1115290}, the uniqueness is a consequence of the convexity of $\mathcal F$. Finally, by standard elliptic regularity, $\phi_\star=(-\Delta_x)^{-1}\,\rho_\star$ is continuous and has a limit as $|x|\to+\infty$.

In dimension $d=1$ or $d=2$, the same scheme can be adapted after proving that $\mathcal F$ is bounded from below. This has been established in~\cite[Theorem~3.4]{DL2019} (also see~\cite{Nernst-Planck}) when $d=2$ under Assumption~\eqref{Hyp-V2}. The case $d=1$ can be dealt with by elementary methods, as follows. Let us consider the potential
\[
V_0(x)=\frac M2\((x+1)\,\mathbb 1_{(-\infty,-1)}(x)+\frac12\,(x+1)\,(1-x)\,\mathbb 1_{[-1,+1]}(x)-(x-1)\,\mathbb 1_{(+1,+\infty)}(x)-\frac12\)
\]
such that, a.e., $-V_0''=\frac M2\,\mathbb 1_{[-1,+1]}=:\rho_0$ and let $\psi:=\phi-V_0$. We claim that
\[
\mathcal F[\rho]=\ir{\rho\,\log\rho}+\ir{\rho\,(V+V_0)}-\frac12\ir{\psi''\,\psi}+\frac12\ir{\rho_0\,\psi}-\frac 12\ir{\rho\,V_0}
\]
is bounded from below because the first two integrals can be bounded using Jensen's inequality and we have $\ir{\psi''\,\psi}=-\ir{|\psi'|^2}$. The balance between the last to terms is slightly more subtle. Let $x_0\in\R$ be such that $\phi(x_0)=M/2$. Using~\eqref{Rep1d}, we notice that $\phi'(x_0)=0$ and obtain
\[
\phi(x)=-\int_{x_0}^xdy\int_{x_0}^y\rho(z)\,dz
\]
because $\phi(x_0)=\phi_0=0$. A crude consequence is the estimate $-\frac M2\,|x-x_0|\le\phi(x)\le0$ for any $x\in\R$, which shows that
\[
\ir{\rho_0\,\psi}\ge\ir{\rho_0\,V_0}+\frac M4\int_{-1}^{+1}\phi(x)\,dx\ge\ir{\rho_0\,V_0}-\frac{M^2}8\int_{-1}^{+1}|x-x_0|\,dx\,,
\]
that is, a lower bound of the order of $|x_0|$ as $|x_0|\to+\infty$. On the other hand, if $|x_0|>1$, we have
\[
-\ir{\rho\,V_0}\ge-\frac M4\int_{-1}^{+1}\rho(x)\,dx+\int_{|x|\ge|x_0|}\rho\,V_0\,dx\ge-\frac{M^2}2+\frac M2\,V_0(x_0)\,,
\]
which is of the order of $V_0(x_0)$ and dominates $|x_0|$ as $|x_0|\to+\infty$, by Assumption~\eqref{Hyp-V2}. Combining these estimates provides us with the lower bound we need. \end{proof}

\subsection{Some non-trivial Poincar\'e inequalities}\label{Sec:Poincare}

Assume that $V$ is such that~\eqref{Hyp-V1}-\eqref{Hyp-V2} hold. Before considering the case of the measure $e^{-W_\star}dx$ on $\R^d$, with $W_\star=V+\phi_\star$, we may ask under which conditions on $V$ the \emph{Poincar\'e inequality}
\be{Ineq:Poincare}
\irdx{|\nabla_xu|^2\,e^{-V}}\ge\mathcal C_{\rm P}\irdx{|u|^2\,e^{-V}}\quad\forall\,u\in\mathrm H^1(\R^d)\quad\mbox{such that}\quad\irdx{u\,e^{-V}}=0
\ee
is true for some constant $\mathcal C_{\rm P}>0$. Let us define $w=u\,e^{-V/2}$ and observe that~\eqref{Ineq:Poincare} is equivalent to
\[
\irdx{|\nabla_xw|^2}+\irdx{\Phi\,|w|^2}\ge\mathcal C_{\rm P}\irdx{|w|^2}
\]
under the condition that $\irdx{w\,e^{-V/2}}=0$. Here $\Phi=\frac14\,|\nabla_xV|^2-\frac12\,\Delta_xV$ is obtained by expanding the square in $\irdx{\big|\nabla_xw+\frac12\,w\,\nabla_xV\big|^2}$ and integrating by parts the cross-term. From the expression of the square, we learn that the kernel of the Schr\"odinger operator $-\Delta_x+\Phi$ on $\mathrm L^2\big(\R^d,dx\big)$ is generated by $e^{-V/2}$. According to Persson's result~\cite[Theorem~2.1]{MR0133586}, the lower end~$\sigma$ of the continuous spectrum of the Schr\"odinger operator $-\,\Delta_x+\Phi$ is such that
\[
\sigma\ge\lim_{r\to+\infty}\mathop{\mathrm{infess}}_{x\in B_r^c}\,\Phi(x)=:\sigma_0\,.
\]
As a consequence, if $\sigma$ is positive, either there is no eigenvalue in the interval $(0,\sigma)$ and $\mathcal C_{\rm P}=\sigma$, or $\mathcal C_{\rm P}$ is the lowest positive eigenvalue, and it is positive by construction. In both cases, we know that~\eqref{Ineq:Poincare} holds for some $\mathcal C_{\rm P}>0$ if $\sigma_0>0$. In order to prove~\eqref{Ineq:Poincare}, it is enough to check that
\be{Hyp-V3a}\tag{V3a}
\sigma_V:=\lim_{r\to+\infty}\mathop{\mathrm{infess}}_{x\in B_r^c}\(\frac14\,|\nabla_xV|^2-\frac12\,\Delta_xV\)>0\quad\mbox{and}\quad\lim_{r\to+\infty}\mathop{\mathrm{infess}}_{x\in B_r^c}|\nabla_xV|>0\,.
\ee
See for instance~\cite{Bakry_2008,Bouin_2020} for further considerations on Assumption~\eqref{Hyp-V3a}. Now let us consider the measure $\rho_\star\,dx=e^{-W_\star}dx$ on $\R^d$ and establish the corresponding Poin\-ca\-r\'e inequality.
\begin{lemma}\label{Lem:Poincare} Assume that $d\ge1$ and consider $V$ such that~\eqref{Hyp-V1}, \eqref{Hyp-V2} and \eqref{Hyp-V3a} hold. We further assume that
\be{Hyp-V4}\tag{V4}
\lim_{r\to+\infty}\mathop{\mathrm{infess}}_{|x|>r}\(\(M-2V'\)^2-2\,V''\)>0\quad\mbox{if}\quad d=1\,.
\ee
If $\phi_\star$ solves~\eqref{Eqn:Poisson-Boltzmann} and $W_\star=V+\phi_\star$, then there is a positive constant $\mathcal C_\star$ such that
\be{Ineq:Poincare2}
\irdx{|\nabla_xu|^2\,\rho_\star}\ge\mathcal C_\star\irdx{|u|^2\,\rho_\star}\quad\forall\,u\in\mathrm H^1(\R^d)\quad\mbox{s.t.}\quad\irdx{u\,\rho_\star}=0\,.
\ee\end{lemma}
\begin{proof} In order to adapt the result for $V$ to $W_\star$, it is enough to prove that
\[
\sigma_{W_\star}:=\lim_{r\to+\infty}\mathop{\mathrm{infess}}_{x\in B_r^c}\(\frac14\,\big|\nabla_x\phi_\star+\nabla_xV\big|^2-\frac12\(\Delta_x\phi_\star+\Delta_xV\)\)>0\,.
\]
By Lemma~\ref{Lem:Poisson-Boltzmann} and~\eqref{Hyp-V3a}, $|\Delta_x\phi_\star|=\rho_\star=o\(|\nabla_xV|^2-2\,\Delta_xV\)$. Using $\nabla_x\phi_\star=\nabla G_d*\rho_\star$, we obtain that $|\nabla_x\phi_\star|=\mathcal O\(|x|^{1-d}\)$ is negligible compared to $|\nabla_xV|$ if $d\ge2$. If $d=1$, the result follows from~\eqref{Hyp-V4} using the fact that $\phi_\star'(x)\sim\pm\,M/2$ as $x\to\mp\infty$.\end{proof}

We shall now replace~\eqref{Hyp-V3a} by the slightly stronger assumption that, for some $\theta\in[0,1)$, 
\be{Hyp-V3b}\tag{V3b}
\lim_{r\to+\infty}\mathop{\mathrm{infess}}_{x\in B_r^c}\(\frac\theta4\,|\nabla_xV|^2-\frac12\,\Delta_xV\)\ge0\quad\mbox{and}\quad\lim_{r\to+\infty}\mathop{\mathrm{infess}}_{x\in B_r^c}|\nabla_xV|>0\,.
\ee
\begin{corollary}\label{Cor:Poincare} Assume that $d\ge1$ and consider $V$ such that~\eqref{Hyp-V1}, \eqref{Hyp-V2}, \eqref{Hyp-V3b} and \eqref{Hyp-V4} hold. If~$\phi_\star$ solves~\eqref{Eqn:Poisson-Boltzmann} and $W_\star=V+\phi_\star$, then there is a positive constant $\mathcal C$ such that
\be{Ineq:Poincare3}
\irdx{|\nabla_xu|^2\,\rho_\star}\ge\mathcal C\irdx{|u|^2\,|\nabla_xW_\star|^2\,\rho_\star}\quad\forall\,u\in\mathrm H^1(\R^d)\quad\mbox{s.t.}\quad\irdx{u\,\rho_\star}=0\,.
\ee\end{corollary}
\begin{proof} By expanding $\big|\nabla_x\(u\,\sqrt{\rho_\star}\)\big|^2$, using $\nabla_x\sqrt{\rho_\star}=-\frac12\,\nabla_xW_\star\,\sqrt{\rho_\star}$ and integrating by parts, we obtain that
\[
0\le\irdx{\big|\nabla_x\(u\,\sqrt{\rho_\star}\)\big|^2}=\irdx{|\nabla_xu|^2\,\rho_\star}-\irdx{\(\frac14\,|\nabla_xW_\star|^2-\frac12\,\Delta_xW_\star\)|u|^2\,\rho_\star}\,.
\]
Combined with~\eqref{Ineq:Poincare2}, this shows that
\begin{multline*}
\irdx{|\nabla_xu|^2\,\rho_\star}\ge\irdx{\left[(1-\eta)\,\mathcal C_\star+\eta\(\frac\theta4\,|\nabla_xW_\star|^2-\frac12\,\Delta_xW_\star\)\right]\,|u|^2\,\rho_\star}\\
+\frac\eta4\,(1-\theta)\irdx{|u|^2\,|\nabla_xW_\star|^2\,\rho_\star}
\end{multline*}
for any $\eta\in(0,1)$. With $\eta$ chosen small enough so that $(1-\eta)\,\mathcal C_\star+\eta\,\big(\frac\theta4\,|\nabla_xW_\star|^2-\frac12\,\Delta_xW_\star\big)$ is nonnegative a.e., the conclusion holds with $\mathcal C=\eta\,(1-\theta)/4$.\end{proof}

In the same spirit as for Corollary~\ref{Cor:Poincare}, we shall assume that for some $\theta\in[0,1)$, 
\begin{multline}\label{Hyp-V5}\tag{V5}
\lim_{r\to+\infty}\mathop{\mathrm{infess}}_{x\in B_r^c}\(\frac\theta4\,|\nabla_xV|^4-\frac12\,\Delta_xV\,|\nabla_xV|^2-\mathrm{Hess}(V):\nabla_xV\otimes\nabla_xV\)\ge0\\
\mbox{and}\quad\lim_{r\to+\infty}\mathop{\mathrm{infess}}_{x\in B_r^c}|\nabla_xV|>0\,.
\end{multline}
\begin{corollary}\label{Cor:Poincare2} Assume that $d\ge1$ and consider $V$ such that~\eqref{Hyp-V1}, \eqref{Hyp-V2}, \eqref{Hyp-V3b} and \eqref{Hyp-V5} hold. If~$\phi_\star$ solves~\eqref{Eqn:Poisson-Boltzmann} and $W_\star=V+\phi_\star$, then there is a positive constant $\mathcal C_\circ$ such that
\[
\irdx{|\nabla_xu|^2\,|\nabla_xW_\star|^2\,\rho_\star}\ge\mathcal C_\circ\irdx{|u|^2\,|\nabla_xW_\star|^4\,\rho_\star}\quad\forall\,u\in\mathrm H^1(\R^d)\quad\mbox{s.t.}\quad\irdx{u\,\rho_\star}=0\,.
\]\end{corollary}
The proof is again based on the expansion of the square in $\big|\nabla_x\(u\,\sqrt{\rho_\star}\)\big|^2\,|\nabla_xW_\star|^2$, integrations by parts and an IMS (for Ismagilov, Morgan, Morgan-Simon, Sigal) truncation argument in order to use Lemma~\ref{Lem:Poincare} in a finite centered ball of radius $2R$, on which $\nabla_xW_\star$ is bounded and Assumption~\eqref{Hyp-V5} outside of the centered ball of radius $R$. See~\cite{MR526292, MR708966} or~\cite[section~2]{MR2379440} for details on the IMS truncation method. Compared with Corollary~\ref{Cor:Poincare}, there is no deeper difficulty and we shall skip further details.

\subsection{Further inequalities based on pointwise estimates}\label{Sec:Pointwise}

If $\mathfrak M$ is a $d\times d$ symmetric real valued matrix, let us denote by $\Lambda(\mathfrak M)$ the largest eigenvalue of $\mathfrak M$. With this notation, let us assume that
\be{Hyp-V6}\tag{V6}
\Lambda_V:=\lim_{r\to+\infty}\mathop{\mathrm{supess}}_{x\in B_r^c}\;\frac1{|\nabla_xV(x)|^2}\,\Lambda\(e^{V(x)}\(\mathrm{Hess}\big(e^{-V(x)}\big)-\frac 12\,\Delta_x\big(e^{-V(x)}\big)\,\Id\)\)<+\infty\,.
\ee
In other words, Assumption~\eqref{Hyp-V6} means that for any $\ve>0$, there exists some $R>0$ such that
\[
e^{V(x)}\(\mathrm{Hess}\big(e^{-V(x)}\big)-\frac 12\,\Delta_x\big(e^{-V(x)}\big)\,\Id\)\le(\Lambda_V-\ve)\,|\nabla_xV(x)|^2\,\Id\,,\quad x\in\R^d\;\mbox{a.e.}\;\mbox{such that}\;|x|>R\,,
\]
where the inequality holds in the sense of positive matrices.
\begin{lemma}\label{Lem:PoincareLambda} Assume that $d\ge1$ and consider $V$ such that~\eqref{Hyp-V1}, \eqref{Hyp-V2} and \eqref{Hyp-V6} hold. If $\phi_\star$ solves~\eqref{Eqn:Poisson-Boltzmann} and $W_\star=V+\phi_\star$, then there is a positive constant $\Lambda_\star$ such that
\[\label{Ineq:PoincareLambda}
\irdx{\(\mathrm{Hess}(\rho_\star)-\frac 12\,\Delta_x\rho_\star\,\Id\):\nabla_xw\otimes\nabla_xw}\le\Lambda_\star\irdx{|\nabla_xw|^2\,|\nabla_xW_\star|^2\,\rho_\star}
\]
for any function $w\in\mathrm H^1_{\rm{loc}}(\R^d)$.\end{lemma}
\begin{proof} An elementary computation shows that
\[
\mathrm{Hess}(\rho_\star)=\big(\nabla_xW_\star\otimes\nabla_xW_\star-\mathrm{Hess}(W_\star)\big)\,\rho_\star\quad\mbox{and}\quad\Delta_x\rho_\star=\(|\nabla_xW_\star|^2-\Delta_xW_\star\)\rho_\star\,.
\]
The proof is then similar to the above arguments, up to elementary estimates, that we shall omit here.\end{proof}

Similarly, let us assume that
\be{Hyp-V7}\tag{V7}
\lim_{r\to+\infty}\mathop{\mathrm{supess}}_{x\in B_r^c}\Big|\nabla_x\(\log\big(|\nabla_xV(x)|^2\big)\)\Big|<+\infty\,.
\ee
\begin{lemma}\label{Lem:PoincareLambda2} Assume that $d\ge1$ and consider $V$ such that~\eqref{Hyp-V1}, \eqref{Hyp-V2} and \eqref{Hyp-V7} hold. If $\phi_\star$ solves~\eqref{Eqn:Poisson-Boltzmann} and $W_\star=V+\phi_\star$, then there is a positive constant $\Lambda_\circ$ such that
\be{Ineq:PoincareLambda2}
\Big|\nabla_x\(|\nabla_xW_\star(x)|^2\)\Big|\le\Lambda_\circ\,|\nabla_xW_\star(x)|^2\,,\quad x\in\R^d\;\mbox{a.e.}\;\mbox{such that}\;|x|>R\,.
\ee\end{lemma}
Here we mean that $\nabla_x\(|\nabla_xW_\star|^2\)\nabla_xw=2\,\mathrm{Hess}(W_\star):\nabla_xW_\star\otimes\nabla_xw$ for any $C^1$ function~$w$. Inequality~\eqref{Ineq:PoincareLambda2} follows from the regularity and decay estimates of $\phi_\star$. Since the proof relies only on elementary but tedious computations, we omit it here. In the same vein, let us assume that
\be{Hyp-V8}\tag{V8}
\left\|\,|\nabla_xV|^2\,e^{-V}\right\|_{\mathrm L^\infty(\R^d,dx)}<+\infty\quad\mbox{and}\quad\big\|\,\big|\nabla_x\(|\nabla_xV|^2\)\big|^2e^{-V}\big\|_{\mathrm L^\infty(\R^d,dx)}<+\infty\,.
\ee
\begin{lemma}\label{Lem:PoincareLambda3} Assume that $d\ge1$ and consider $V$ such that~\eqref{Hyp-V1}, \eqref{Hyp-V2} and \eqref{Hyp-V8} hold. If $\phi_\star$ solves~\eqref{Eqn:Poisson-Boltzmann} and $W_\star=V+\phi_\star$, then $\|\,|\nabla_xW_\star|^2\,\rho_\star\|_{\mathrm L^\infty(\R^d,dx)}$ and $\big\|\,\big|\nabla_x\(|\nabla_xW_\star|^2\)\big|^2\rho_\star\big\|_{\mathrm L^\infty(\R^d,dx)}$ are finite.\end{lemma}

\subsection{A Bochner-Lichnerowicz-Weitzenb\"ock identity and second order estimates}

Algebraic computations and a few integrations by parts provide us with the following estimate.
\begin{lemma}\label{Lem:BLW} Let $\rho_\star=e^{-W_\star}$ be  a non-trivial function in $\mathrm L^\infty_{\rm{loc}}\cap\mathrm W^{1,2}(\R^d)$. Then for any smooth function $w$ on~$\R^d$ with compact support, we have the identity
\[
\irdx{|\mathrm{Hess}(w)|^2\,\rho_\star}\le6\irdx{\frac1{\rho_\star}\left|\nabla_x\cdot\(\rho_\star\,\nabla_xw\)\right|^2}+8\irdx{\big(\nabla_xW_\star\cdot\nabla_xw\big)^2\,\rho_\star}\,.
\]
\end{lemma}
Notice that if $V$ satisfies~\eqref{Hyp-V1}--\eqref{Hyp-V2}, $M>0$ and $W_\star=V+\phi_\star$ where $\phi_\star$ is the unique solution of~\eqref{Eqn:Poisson-Boltzmann}, then $\rho_\star$ is an admissible function for Lemma~\ref{Lem:BLW}.
\begin{proof} Let us start by establishing a Bochner-Lichnerowicz-Weitzenb\"ock identity as follows:
\begin{align*}
\frac 12\,\Delta_x\(\rho_\star\,|\nabla_xw|^2\)&=\nabla_x\cdot\(\rho_\star\,\mathrm{Hess}(w)\,\nabla_xw\)+\frac12\,\nabla_x\cdot\(|\nabla_xw|^2\,\nabla_x\rho_\star\)\\
&=\rho_\star\,|\mathrm{Hess}(w)|^2+\rho_\star\,\nabla_xw\cdot\nabla_x(\Delta_xw)+\frac12\,\Delta_x\rho_\star\,|\nabla_xw|^2\\
&\hspace*{2.8cm}+2\,\mathrm{Hess}(w):\nabla_xw\otimes\nabla_x\rho_\star\\
&=\rho_\star\,|\mathrm{Hess}(w)|^2+\nabla_xw\cdot\nabla_x(\rho_\star\,\Delta_xw)-(\nabla_xw\cdot\nabla_x\rho_\star)\,\Delta_xw\\
&\hspace*{2.8cm}+\frac12\,\Delta_x\rho_\star\,|\nabla_xw|^2+2\,\mathrm{Hess}(w):\nabla_xw\otimes\nabla_x\rho_\star\,.
\end{align*}
We obtain after a few integrations by parts on $\R^d$ that
\begin{align*}
&\irdx{\Delta_x\(\rho_\star\,|\nabla_xw|^2\)}=0\,,\quad\irdx{\nabla_xw\cdot\nabla_x(\rho_\star\,\Delta_xw)}=-\irdx{(\Delta_xw)^2\,\rho_\star}\,,\\
&\frac12\irdx{\Delta_x\rho_\star\,|\nabla_xw|^2}+\irdx{\mathrm{Hess}(w):\nabla_xw\otimes\nabla_x\rho_\star}=0\,,
\end{align*}
which proves that
\begin{multline*}\label{BLW}
\irdx{|\mathrm{Hess}(w)|^2\,\rho_\star}\\
=\irdx{(\Delta_xw)^2\,\rho_\star}+\irdx{(\nabla_xw\cdot\nabla_x\rho_\star)\,\Delta_xw}-\irdx{\mathrm{Hess}(w):\nabla_xw\otimes\nabla_x\rho_\star}\,.
\end{multline*}
We deduce from
\begin{multline*}
\irdx{(\nabla_xw\cdot\nabla_x\rho_\star)\,\Delta_xw}=-\irdx{\Delta_xw\,(\nabla_xw\cdot\nabla_xW_\star)\,\rho_\star}\\
\le\frac12\irdx{(\Delta_xw)^2\,\rho_\star}+\frac12\irdx{\big(\nabla_xW_\star\cdot\nabla_xw\big)^2\,\rho_\star}
\end{multline*}
and 
\begin{multline*}
-\irdx{\mathrm{Hess}(w):\nabla_xw\otimes\nabla_x\rho_\star}=\irdx{\mathrm{Hess}(w):\nabla_xw\otimes\nabla_xW_\star\,\rho_\star}\\
\le\frac12\irdx{(\mathrm{Hess}(w))^2\,\rho_\star}+\frac12\irdx{\big(\nabla_xW_\star\cdot\nabla_xw\big)^2\,\rho_\star}
\end{multline*}
that
\[
\frac12\irdx{|\mathrm{Hess}(w)|^2\,\rho_\star}\leq\frac32\irdx{(\Delta_xw)^2\,\rho_\star}+\irdx{\big(\nabla_xW_\star\cdot\nabla_xw\big)^2\,\rho_\star}\,.
\]
Since $\nabla_x\rho_\star=-\nabla_xW_\star\,\rho_\star$ and $\Delta_xw\,\rho_\star=\nabla_x\cdot\(\rho_\star\,\nabla_xw\)+\(\nabla_xW_\star\cdot\nabla_xw\)\rho_\star$, we have the estimate
\[
\irdx{(\Delta_xw)^2\,\rho_\star}\le2\irdx{\frac1{\rho_\star}\left|\nabla_x\cdot\(\rho_\star\,\nabla_xw\)\right|^2}+2\irdx{\big(\nabla_xW_\star\cdot\nabla_xw\big)^2\,\rho_\star}\,,
\]
which completes the proof.\end{proof}

\subsection{The scalar product}\label{Sec:ScalarProduct}

On $\R^d\times\R^d$, let us define the measure
\[
d\mu:=f_\star(x,v)\,dx\,dv
\]
and consider the functional space
\be{Hilbert}
\mathcal H:=\left\{h\in\mathrm L^1\cap\mathrm L^2\(\R^d\times\R^d,d\mu\)\,:\,\irdmu h=0\;\mbox{and}\,\irdx{|\nabla_x\psi_h|^2}<\infty\right\}\,,
\ee
where we use the notation $\rho_h=\irdv{h\,f_\star}$ and $\psi_h=(-\Delta_x)^{-1}\,\rho_h$. We also define
\be{Hilbert-scalar}
\scalar{h_1}{h_2}:=\irdmu{h_1\,h_2}+\irdx{\rho_{h_1}\,(-\Delta_x)^{-1}\rho_{h_2}}\quad\forall\,h_1,\,h_2\in\mathcal H\,.
\ee
\begin{remark}\label{Rem:FreeEnergy} This definition deserves an explanation. The whole motivation of this paper is to understand the large-time asymptotics of the nonlinear system~\eqref{Syst:VPFP} and it is known, see for instance~\cite{MR1306570}, that the irreversibility and the convergence of a solution $f$ to the stationary solution $f_\star$ can be studied using the \emph{free energy} functional
\[
\irdxv{f\,\log\(\frac f{f_\star}\)}+\frac12\irdx{|\nabla_x\phi|^2}\,.
\]
Indeed, if $(f,\phi)$ solves~\eqref{Syst:VPFP} and has sufficient smoothness and decay properties, we have
\[
\frac d{dt}\(\irdxv{f\,\log\(\frac f{f_\star}\)}+\frac12\irdx{|\nabla_x\phi|^2}\)=-\irdxv{f^{-1}\,|\nabla_vf|^2}\,.
\]
However, no rate of decay of the free energy is known for initial data such that $f-f_\star$ is large. With $f=f_\star\,(1+\veta\,h)$ for some $h$ such that $\irdxv{h\,f_\star}=0$, we can at least investigate the limit as $\eta\to0_+$. At leading order $\eta^2$, the free energy is $\veta^2\scalar hh$, which is the initial reason for introducing this scalar product.
\end{remark}
\begin{lemma}\label{Lem:Hilbert} Let $M>0$. If $V$ satisfies~\eqref{Hyp-V1}--\eqref{Hyp-V2}, then $\big(\mathcal H,\scalar\cdot\cdot\big)$ is a Hilbert space for any dimension $d\ge1$.\end{lemma}
\begin{proof} Up to an integration by parts, we can rewrite $\scalar{h_1}{h_2}$ as
\[
\scalar{h_1}{h_2}=\irdmu{h_1\,h_2}+\irdx{(-\Delta_x\psi_{h_1})\,\psi_{h_2}}=\irdmu{h_1\,h_2}+\irdx{\nabla_x\psi_{h_1}\cdot\nabla_x\psi_{h_2}}
\]
and observe that this determines a scalar product. This computation has to be justified. Let us distinguish three cases depending on the dimension $d$.

\medskip Let us assume first that $d\ge3$. We know that $\psi_\star=G_d\ast\rho_\star$ is nonnegative and deduce that $\rho_\star$ is bounded because
\[
0\le e^{-V-\psi_\star}\le e^{-V}\in\mathrm L^\infty(\R^d)\,.
\]
Hence, for any $p\in(1,2]$, we have
\[
\nrm{\rho_h}{\mathrm L^p(\R^d)}^p=\int_{\R^d}\left|\int_{\R^d}h\,f_\star\,dv\right|^p\,dx\le\nrm{\rho_\star}{\mathrm L^\infty(\R^d)}^{p-1}\,\irdmu{|h|^p}\,.
\]
According to~\cite{Lieb-83}, we know by the Hardy-Littlewood-Sobolev inequality that
\[
\int_{\R^d\times\R^d}\frac{|\rho_1(x)|\,|\rho_2(x)|}{|x-y|^{d-a}}\,dx\,dy\le\mathcal C_{\mathrm{HLS}}\nrm{\rho_1}{\mathrm L^p(\R^d)}\,\nrm{\rho_2}{\mathrm L^q(\R^d)}
\]
if $a\in(0,d)$ and $p$, $q\in(1,+\infty)$ are such that $1+\frac ad=\frac1p+\frac1q$. This justifies the fact that $\irdx{\rho_h\,(-\Delta_x)^{-1}\rho_h}$ is well defined if $h\in\mathrm L^1\cap\mathrm L^2\(\R^d\times\R^d,d\mu\)$. With $a=2$, $p<3/2$ if $d=3$, $p<2$ if $d=4$ and $p\le2$ if $d\ge5$, we deduce that $\psi_h\in\mathrm L^{q'}(\R^d)$ where $q'=q/(q-1)=d\,p/(d-2\,p)$. A simple H\"older estimate shows the Gagliardo-Nirenberg type estimate
\[
\nrm{\nabla_x\psi}{\mathrm L^2(\R^d)}^2\le\nrm{\Delta_x\psi}{\mathrm L^{p_1}(\R^d)}\nrm\psi{\mathrm L^{q_1}(\R^d)}
\]
and proves for an appropriate choice of $(p_1,q_1)\in(1,2)\times(2,+\infty)$ with $\frac1{p_1}+\frac1{q_1}=1$ that $\nabla_x\psi_h$ is bounded in $\mathrm L^2(\R^d)$.

\medskip The case $d=2$ is well known. The boundedness of $\nrm{\rho_h}{\mathrm L^p(\R^d)}$ for any $p\in(1,2]$ follows by the same argument as in the case $d\ge3$ and we learn that $|\rho_h|\,\log|\rho_h|$ is integrable by log-H\"older interpolation. The boundedness from below of $\int_{\R^2}{\rho_h\,(-\Delta_x)^{-1}\rho_h}$ is then a consequence of the logarithmic Hardy-Littlewood-Sobolev inequality, see~\cite{MR1143664,DL2019}. Using the fact that $\irdx{\rho_h}=0$, we also know from~\cite{MR2226917} that $\nabla_x\psi_h$ is bounded in $\mathrm L^2(\R^2)$.

\medskip When $d=1$, the nonnegativity of the scalar product is a consequence of~\eqref{Nonnegativity1d} and holds without additional condition by a simple density argument.\end{proof}

The condition $\irdmu h=0$ in the definition of $h$ is simply an orthogonality condition with the constant functions, with respect to the usual scalar product in $\mathrm L^2\(\R^d\times\R^d,d\mu\)$. By taking the completion of smooth compactly supported functions with zero average with respect to the norm defined by $h\mapsto\scalar hh$, we recover $\mathcal H$, which is therefore a Hilbert space. In the next sections, we shall denote by $\|\cdot\nH$ the norm on $\mathcal H$ associated with the scalar product so that
\[
\|h\nH^2=\scalar hh\quad\forall\,h\in\mathcal H\,.
\]
Whenever another norm is used, this will be explicitly specified.

\section{Proof of the main result}\label{Sec:Proof}

In this section, we prove Theorem~\ref{Thm:Main}. Our task is to check that the assumptions of Section~\ref{Sec:Abstract} hold in the functional setting of Section~\ref{Sec:FunctionalSetting}.

\subsection{Definitions and elementary properties}

On the space $\mathcal H$ defined by~\eqref{Hilbert} with scalar product given by~\eqref{Hilbert-scalar}, let us consider the transport and the collision operators respectively defined by
\be{TL}
\sfT h:=v\cdot\nabla_xh-\nabla_xW_\star\cdot\nabla_vh+v\cdot\nabla_x\psi_h\,,\quad\sfL h:=\Delta_vh-v\cdot\nabla_vh
\ee
where $W_\star=V+\phi_\star$. In the literature, $\sfL$ is known as the \emph{Ornstein-Uhlenbeck operator}.
\begin{lemma}\label{Lem:Operators} With the above notation, $\sfL$ and $\sfT$ are respectively symmetric and anti-sym\-metric operators on $(\mathcal{ H},\langle,\rangle)$.\end{lemma}
The operator $\sfL$ is defined on continuous functions with compact support which are of class $C^2$ with respect to $v$ and it is semi-bounded. It can be extended using Friederichs extension as a self-adjoint operator on $\mathcal H$. Concerning $\sfT$, we shall define it on the space of continuous functions with compact support of class $C^1$ with respect to $x$ and $v$ and omit details concerning domain issues and extensions as we need only properties that apply to solutions of the evolution problem~\eqref{Syst:VPFPlin}.
\begin{proof} If $h_1$ and $h_2$ are two functions in $\mathrm W^{2,2}(\R^d,\mathcal M\,dv)$, then $\sfL$ is such that
\[
\irh{(\sfL h_1)\,h_2\,\mathcal M}=-\irh{\nabla_vh_1\cdot\nabla_vh_2\,\mathcal M}
\]
and as a special case corresponding to $h_1=h$, $h_2=1$, we find that $\rho_{\sfL h}=\irh{(\sfL h)\,f_\star}=0$ and also $\psi_{\sfL h}=0$ for any $h\in\mathcal H$. As a consequence, we have that
\[
\scalar{(\sfL h_1)}{h_2}=-\irdmu{\nabla_vh_1\cdot\nabla_vh_2}=\scalar{h_1}{(\sfL h_2)}\,.
\]
Concerning the transport operator, we know that $\sfT f_\star=0$. Hence an integration by parts shows that
\[
\scalar{(\sfT h_1)}{h_2}=\irdmu{\(v\cdot\nabla_xh_1-\nabla_xW_\star\cdot\nabla_vh_1\)h_2}=-\,\scalar{h_1}{(\sfT h_2)}
\]
for function $h_1$ and $h_2$ in $\mathcal M$ which are smooth enough, because $\rho_{\sfT h}=\irh{(\sfT h)\,f_\star}=\nabla_x\psi_h\cdot\irh{v\,f_\star}=0$ and $\psi_{\sfT h}=0$.\end{proof}

\subsection{Microscopic coercivity}\label{Sec:Micro}

By the Gaussian Poincar\'e inequality, we know that
\[
\irdv{|\nabla_vg|^2\,\mathcal M}\ge\irdv{\left|g-\sfPi g\right|^2\,\mathcal M}\quad\forall\,g\in\mathrm H^1\(\R^d,\,\mathcal M\,dv\)\,,
\]
where $\sfPi g=\irdv{g\,\mathcal M}$ denotes the average of $g$ with respect to the Gaussian probability measure $\mathcal M\,dv$. By extension, we shall consider $\sfPi$ as an operator on $\mathcal H$ and observe that
\be{uh}
\sfPi h=u_h:=\frac{\rho_h}{\rho_\star}=\frac{\irh{h\,f_\star}}{\irh{f_\star}}=\irdv{h\,\mathcal M}\quad\forall\,h\in\mathcal H\,.
\ee
Let us notice first that $\sfPi$ is an orthogonal projector.
\begin{lemma}\label{Lem:Projection}$\sfPi$ is a self-adjoint operator and $\sfPi\circ\sfPi=\sfPi$.\end{lemma}
\begin{proof} It is elementary to check that
\[
(\sfPi\circ\sfPi)\,h=\sfPi u_h=u_h\,,\quad\irdmu{(\sfPi h_1)\,h_2}=\ird{u_{h_1}\,u_{h_2}\,\rho_\star}
\]
and
\[
\irdx{\rho_{\sfPi h_1}\,(-\Delta_x)^{-1}\rho_{h_2}}=\irdx{\rho_{h_1}\,(-\Delta_x)^{-1}\rho_{h_2}}
\]
because $\rho_{h_1}=\rho_\star\,u_{h_1}=\rho_\star\,u_{\sfPi h_1}=\rho_{\sfPi h_1}$.\end{proof}
\begin{lemma}\label{Lem:MicroscopicCoercivity} Microscopic coercivity~\eqref{H1} holds with $\lambda_m=1$.\end{lemma}
\begin{proof} We already know that $-\,\scalar{(\sfL h)}h=\irdmu{|\nabla_vh|^2}$ and $\rho_{h-\sfPi h}=\rho_h-\rho_{\sfPi h}=0$ so that
\[
\left\|h-\sfPi h\right\|^2=\nrm{h-\sfPi h}{\mathrm L^2(\R^d\times\R^d,d\mu)}^2=\irdmu{\left|h-\sfPi h\right|^2}\,.
\]
The conclusion is then a consequence of the Gaussian Poincar\'e inequality.
\end{proof}

\subsection{Macroscopic coercivity}\label{Sec:Macro}

\begin{lemma}\label{Lem:MacroscopicCoercivity} Assume that $d\ge1$ and consider $V$ such that~\eqref{Hyp-V1}, \eqref{Hyp-V2}, \eqref{Hyp-V3a} and \eqref{Hyp-V4} hold. With the notations of Lemma~\ref{Lem:Poincare}, macroscopic coercivity~\eqref{H2} holds with $\lambda_M=\mathcal C_\star$.\end{lemma}
\begin{proof} Using $\sfT\sfPi h=v\cdot\(\nabla_xu_h+\nabla_x\psi_h\)$ with $u_h$ as in~\eqref{uh}, $\irh{(v\cdot\mathsf e)^2\,\mathcal M}=1$ for any given $\mathsf e\in\mathbb S^{d-1}$ and~\eqref{Ineq:Poincare2}, we find that
\[
\left\|\sfT\sfPi h\right\nH^2=\irdx{|\nabla_xu_h+\nabla_x\psi_h|^2\,\rho_\star}\ge\mathcal C_\star\left[\irdx{|u_h+\psi_h|^2\,\rho_\star}-\frac1M\(\irdx{\psi_h\,\rho_\star}\)^2\right]
\]
because $\irdx{u_h\,\rho_\star}=\irdx{\rho_h}=0$. We know from Lemma~\ref{Lem:Hilbert} that $\irdx{u_h\,\psi_h\,\rho_\star}=\irdx{\rho_h\,\psi_h}\ge0$ and by the Cauchy-Schwarz inequality, we get that
\[
\(\irdx{\psi_h\,\rho_\star}\)^2\le M\,\irdx{|\psi_h|^2\,\rho_\star}\,.
\]
Altogether, we collect these estimates into
\[
\irdx{|\nabla_xu_h+\nabla_x\psi_h|^2\,\rho_\star}\ge\mathcal C_\star\left[\irdx{|u_h|^2\,\rho_\star}+\irdx{\rho_h\,\psi_h}\right]=\mathcal C_\star\,M\,\|u_h\nH^2\,,
\]
which concludes the proof.\end{proof}

\subsection{Parabolic macroscopic dynamics}\label{Sec:ParabolicMacroDyn}

\begin{lemma}\label{Lem:ParabolicMacroDyn} The transport operator $\sfT$ satisfies the parabolic macroscopic dynamics~\eqref{H3}.\end{lemma}
\begin{proof} Since $\sfT\sfPi h=v\cdot\(\nabla_xu_h+\nabla_x\psi_h\)$, we obtain that
\[
\sfPi\sfT\sfPi h\,\rho_\star=\(\nabla_xu_h+\nabla_x\psi_h\)\cdot\irdv{v\,f_\star}=0\,.
\]\end{proof}

\subsection{Bounded auxiliary operators}\label{Sec:BoundedAuxiliaryOperators}

The point is to prove that~\eqref{H4} holds, \emph{i.e.}, that for any $F\in\mathcal H$, $\|\sfA\sfT(\Id-\sfPi)F\nH$ and $\|\sfA\sfL F\nH$ are bounded up to a constant by $\|(\Id-\sfPi)F\nH$. This is the purpose of Lemma~\ref{Lem:BoundedAuxiliaryOperators} and Lemma~\ref{T(1moinsPi)}. The two quantities, $\|\sfA\sfT(\Id-\sfPi)F\nH$ and $\|\sfA\sfL F\nH$, are needed to control the bad terms in the expression of $\mathsf D_\delta$, in the abstract formulation of Proposition~\ref{DMS}, namely $\scalar{\sfT\sfA F}F$, $\scalar{\sfA\sfT(\Id-\sfPi)F}F$ and $\scalar{\sfA\sfL F}F$ (which have no definite sign), by the two good terms, $-\,\scalar{\sfL F}F$ and $\scalar{\sfA\sfT\sfPi F}F$ (which are both positive).
\begin{lemma}\label{Lem:BoundedAuxiliaryOperators} The operators $\sfT\sfA$ and $\sfA\sfL$ satisfy: for all $h\in\mathrm L^2\big(\R^d\times\R^d,d\mu\big)$,
\[
\|\sfA\sfL\,h\nH\le\frac12\,\|(\Id-\sfPi)h\nH\,.
\]\end{lemma}
\begin{proof} If we denote the flux by $j_h:=\irdv{v\,h\,f_\star}$, we remark that $j_{\sfL h}=-j_h$ and
\[
\sfPi\sfT h=\nabla_x\cdot j_h-\(\nabla_xV+\nabla_x\phi_\star\)\cdot j_h\,.
\]
Since $\sfA h=g$ means $g+(\sfT\sfPi)^\ast(\sfT\sfPi)g=(\sfT\sfPi)^*h=-\,\sfPi\sfT h$, this implies that
\[
\sfA\sfL h=-\,\sfA h\,.
\]
The same computation as for~\eqref{Eqn:AF-TAF} shows that $\|\sfA\sfL h\nH^2=\|\sfA h\nH^2=\|g\nH^2\le\frac14\|(\Id-\sfPi)h\nH^2$, which completes the proof.\end{proof}

\begin{lemma}\label{T(1moinsPi)} Assume that $d\ge1$ and consider $V$ such that~\eqref{Hyp-V1}, \eqref{Hyp-V2}, \eqref{Hyp-V3b}, \eqref{Hyp-V4}, \eqref{Hyp-V5}, \eqref{Hyp-V6}, \eqref{Hyp-V7} and \eqref{Hyp-V8} hold. There exists a constant $\mathcal C>0$ such that
\[
\|\sfA\sfT(\Id-\sfPi)\,h\nH\le\mathcal C\,\|(\Id-\sfPi)\,h\nH\quad\forall\,h\in\mathcal H\,.
\]
\end{lemma}
\begin{proof} In order to get an estimate of $\|\sfA\sfT(\Id-\sfPi)\,h\nH$, we will compute $\|\sfA\sfT(\Id-\sfPi))^*h\nH$.

\medskip\noindent\emph{Step 1: Reformulation of the inequality as an elliptic regularity estimate.} We claim that
\be{Claim0}
\big\|\big(\sfA\sfT(\Id-\sfPi)\big)^*h\nHH=\irdmu{\big|\big(\sfA\sfT(\Id-\sfPi)\big)^*h\big|^2}=2\ird{|\mathrm{Hess}(w_g)|^2\rho_\star}\,,
\ee
where $w_g:=u_g+\psi_g$ and $-\Delta_x\psi_g=\rho_g$ is computed in terms of
\[
g=\big(\Id+(\sfT\sfPi)^\ast(\sfT\sfPi)\big)^{-1}h\,,
\]
which is obtained by solving the elliptic equation
\be{Eqn:g}
g-\Delta_xw_g+\nabla_xW_\star\cdot\nabla_xw_g=h\,.
\ee
Let $u_h=\sfPi h$ and $w_h:=u_h+\psi_h$. We observe that $\sfT\sfPi h=v\cdot\nabla_xw_h$, $\rho_{\sfT\sfPi h}=0$ and, as a consequence
\[
(\sfT\sfPi)^*(\sfT\sfPi)\,h=-\,\sfPi\sfT(\sfT\sfPi\,h)=-\,\Delta_xw_h+\nabla_xW_\star\cdot\nabla_xw_h=-\,e^{W_\star}\,\nabla_x\(e^{-W_\star}\,\nabla_xw_h\)
\]
where $W_\star=V+\phi_\star$ is such that $\rho_\star=e^{-W_\star}$. With $g$ obtained from~\eqref{Eqn:g}, we compute
\begin{multline*}
\big(\sfA\sfT(\Id-\sfPi)\big)^*h=-\,(\Id-\sfPi)\sfT\sfA^*h=-\,(\Id-\sfPi)\sfT(\sfT\sfPi)\big(\Id+(\sfT\sfPi)^\ast(\sfT\sfPi)\big)^{-1}h\\
=-\,(\Id-\sfPi)\sfT(\sfT\sfPi)g=-\,(\Id-\sfPi)\,\big(v\otimes v:\mathrm{Hess}(w_g)\big)=\Delta_xw_g-v\otimes v:\mathrm{Hess}(w_g)
\end{multline*}
where $\mathrm{Hess}(w)=(\nabla_x\otimes\nabla_x)w$ denotes the Hessian of $w$. Hence, with $|\mathrm{Hess}(w)|^2=\mathrm{Hess}(w):\mathrm{Hess}(w)$, we obtain~\eqref{Claim0} using the following elementary computation

Let $\aa=(\aa_{ij})_{i,j=1}^d$ be a symmetric matrix with coefficients which do not depend on $v$. We compute $\mathsf S:=\irdM{\(\aa:v\otimes v-\mathrm{Tr}(\aa)\)^2}$ as follows. Using
\begin{align*}
\(\aa:v\otimes v-\mathrm{Tr}(\aa)\)^2&=\(\sum_{i,j=1}^d\aa_{ij}\,v_i\,v_j-\sum_{i=1}^d\aa_{ii}\)^2\\
&=\(\sum_{i,j=1}^d\aa_{ij}\,v_i\,v_j\)^2-2\(\sum_{i=1}^d\aa_{ii}\)\(\sum_{i,j=1}^d\aa_{ij}\,v_i\,v_j\)+\(\sum_{i=1}^d\aa_{ii}\)^2
\end{align*}
and $\irdM{v_i\,v_j}=\delta_{ij}$, we obtain
\[
\mathsf S=\irdM{\(\sum_{i,j=1}^d\aa_{ij}\,v_i\,v_j\)^2}-\(\sum_{i=1}^d\aa_{ii}\)^2\,.
\]
Since
\[
\(\sum_{i,j=1}^d\aa_{ij}\,v_i\,v_j\)^2=\(\sum_{i\neq j=1}^d\aa_{ij}\,v_i\,v_j\)^2+\(\sum_{i=1}^d\aa_{ii}\,v_i^2\)^2+2\sum_{i\neq j=1}^d\sum_{k=1}^d\aa_{ij}\,\aa_{kk}\,v_i\,v_j\,v_k^2\,,
\]
$\irdM{v_i^2}=1$, and $\irdM{v_i^4}=3$, the computation simplifies to
\begin{align*}
\irdM{\(\sum_{i,j=1}^d\aa_{ij}\,v_i\,v_j\)^2}&=2\sum_{i\neq j=1}^d\aa_{ij}^2+\sum_{i\neq j=1}^d\aa_{ii}\,\aa_{jj}+3\sum_{i=1}^d\aa_{ii}^2\\
&=2\sum_{i,j=1}^d\aa_{ij}^2+\(\sum_{i=1}^d\aa_{ii}\)^2\,.
\end{align*}
Altogether, this proves that
\[
\mathsf S=2\sum_{i,j=1}^d\aa_{ij}^2=2\,|\aa|^2\,.
\]
The result follows with $\aa=\mathrm{Hess}(w_g)$. A bound on $\ird{|\mathrm{Hess}(w_g)|^2\rho_\star}$ will now be obtained by elliptic regularity estimates based on~\eqref{Eqn:g}.

\medskip\noindent\emph{Step 2: Some $\mathrm H^1$-type estimates.} By integrating~\eqref{Eqn:g} against $\mathcal M(v)\,dv$, we notice that
\be{Eqn:g2}
u_g-\frac1{\rho_\star}\,\nabla_x\cdot\big(\rho_\star\,\nabla_xw_g\big)=u_h
\ee
so that
\be{ZeroAverage}
\irdx{u_g\,\rho_\star}=\irdx{u_h\,\rho_\star}=\irdmu h=0\,.
\ee
If we multiply~\eqref{Eqn:g2} by $w_g\,\rho_\star$ and integrate over $\R^d$, we get after an integration by parts that
\[
\irdx{u_g\,(u_g+\psi_g)\,\rho_\star}+\irdx{|\nabla_xw_g|^2\,\rho_\star}=\irdx{u_h\,(u_g+\psi_g)\,\rho_\star}\,.
\]
Using $\irdx{u_g\,\psi_g\,\rho_\star}=\irdx{|\nabla_x\psi_g|^2}$ and $\irdx{u_h\,\psi_g\,\rho_\star}=\irdx{\nabla_x\psi_h\cdot\nabla_x\psi_g}$ on the one hand, and the elementary estimates
\begin{align*}
&\left|\irdx{u_h\,u_g\,\rho_\star}\right|\le\frac12\irdx{\(|u_g|^2+|u_h|^2\)\rho_\star}\,,\\
&\left|\irdx{\nabla_x\psi_h\cdot\nabla_x\psi_g}\right|\le\frac12\irdx{\(|\nabla_x\psi_h|^2+|\nabla_x\psi_g|^2\)}\,,
\end{align*}
on the other hand, we obtain that
\be{Claim1}
\irdx{|u_g|^2\,\rho_\star}+\irdx{|\nabla_x\psi_g|^2}+2\irdx{|\nabla_xw_g|^2\,\rho_\star}\le\|\sfPi\,h\nH^2
\ee
where
\[
\|\sfPi\,h\nH^2=\irdx{|u_h|^2\,\rho_\star}+\irdx{|\nabla_x\psi_h|^2}\,.
\]
Using $|\nabla_xu_g|^2=|\nabla_xw_g-\nabla_x\psi_g|^2\le2\(|\nabla_xw_g|^2+|\nabla_x\psi_g|^2\)$, we deduce from~\eqref{Claim1} that
\be{Claim2}
\irdx{|\nabla_xu_g|^2\,\rho_\star}\le2\irdx{|\nabla_xw_g|^2\,\rho_\star}+2\irdx{|\nabla_x\psi_g|^2\,\rho_\star}\le\mathcal K\,\|\sfPi\,h\nH^2
\ee
with $\mathcal K=1+2\,\|\rho_\star\|_{\mathrm L^\infty(\R^d,dx)}$.

\medskip\noindent\emph{Step 3: Weighted Poincar\'e inequalities and weighted $\mathrm H^1$-type estimates.} The solution $u_g$ of~\eqref{Eqn:g2} has zero average according to~\eqref{ZeroAverage}. We deduce from Corollary~\ref{Cor:Poincare} that
\[
\irdx{|\nabla_xu_g|^2\,\rho_\star}\ge\mathcal C\irdx{|u_g|^2\,|\nabla_xW_\star|^2\,\rho_\star}\,,
\]
from which we get that
\be{ClaimP1}
X_1^2:=\irdx{|u_g|^2\,|\nabla_xW_\star|^2\,\rho_\star}\le\frac{\mathcal K}{\mathcal C}\,\|\sfPi\,h\nH^2\,.
\ee

Next, we look for a similar estimate for $\irdx{|\psi_g|^2\,|\nabla_xW_\star|^2\,\rho_\star}$. The potential $\psi_g$ has generically a non-zero average $\overline\psi_g:=\frac1M\irdx{\psi_g\,\rho_\star}$ which can be estimated by
\begin{multline*}
M^2\,|\overline\psi_g|^2=\(\irdx{\psi_g\,\rho_\star}\)^2=\(\irdx{\psi_g\,(-\Delta_x\phi_\star)}\)^2=\(\irdx{(-\Delta_x\psi_g)\,\phi_\star}\)^2\\
=\(\irdx{u_g\,\phi_\star\,\rho_\star}\)^2\le\irdx{|\phi_\star|^2\,\rho_\star}\irdx{|u_g|^2\,\rho_\star}\le\kappa_1\,\|\sfPi\,h\nH^2
\end{multline*}
with $\kappa_1:=\irdx{|\phi_\star|^2\,\rho_\star}$, using~\eqref{Claim1}. Since $\nabla_x\rho_\star=-\nabla_xW_\star\,\rho_\star$, we also have
\[
\irdx{\psi_g\,|\nabla_xW_\star|^2\,\rho_\star}=-\irdx{\psi_g\,\nabla_xW_\star\cdot\nabla_x\rho_\star}=\irdx{\(\psi_g\,\Delta_xW_\star+\nabla_x\psi_g\cdot\nabla_xW_\star\)\rho_\star}
\]
and, using the Cauchy-Schwarz inequality,
\begin{multline*}
\(\irdx{\psi_g\,|\nabla_xW_\star|^2\,\rho_\star}\)^2\le\irdx{|\psi_g|^2\,\rho_\star}\irdx{(\Delta_xW_\star)^2\,\rho_\star}\\+\irdx{|\nabla_x\psi_g|^2}\,\|\rho_\star\|_{\mathrm L^\infty(\R^d,dx)}\irdx{|\nabla_xW_\star|^2\,\rho_\star}\,.
\end{multline*}
By Lemma~\ref{Lem:Poincare} applied to $\psi_g-\overline\psi_g$, 
\[
\mathcal C_\star\irdx{|\psi_g|^2\,\rho_\star}\le\|\rho_\star\|_{\mathrm L^\infty(\R^d,dx)}\irdx{|\nabla_x\psi_g|^2}+\mathcal C_\star\,|\overline\psi_g|^2\,,
\]
and~\eqref{Claim1}, we conclude that
\[
\(\irdx{\psi_g\,|\nabla_xW_\star|^2\,\rho_\star}\)^2\le\kappa_2\,\|\sfPi\,h\nH^2
\]
where
\[
\kappa_2:=\(\frac1{\mathcal C_\star}\irdx{(\Delta_xW_\star)^2\,\rho_\star}+\irdx{|\nabla_xW_\star|^2\,\rho_\star}\)\|\rho_\star\|_{\mathrm L^\infty(\R^d,dx)}+\frac{\kappa_1}{M^2}\irdx{(\Delta_xW_\star)^2\,\rho_\star}\,.
\]
By applying Corollary~\ref{Cor:Poincare} to $\psi_g-\overline\psi_g$, we deduce from
\[
\mathcal C\irdx{|\psi_g-\overline\psi_g|^2\,|\nabla_xW_\star|^2\,\rho_\star}\le\irdx{|\nabla_x\psi_g|^2\,\rho_\star}
\]
that
\[
\mathcal C\irdx{|\psi_g|^2\,|\nabla_xW_\star|^2\,\rho_\star}\le\irdx{|\nabla_x\psi_g|^2\,\rho_\star}+2\,\mathcal C\,\overline\psi_g\irdx{\psi_g\,\rho_\star\,|\nabla_xW_\star|^2\,\rho_\star}\,.
\]
Hence
\be{ClaimP2}
\irdx{|\psi_g|^2\,|\nabla_xW_\star|^2\,\rho_\star}\le\(\frac{\|\rho_\star\|_{\mathrm L^\infty(\R^d,dx)}}{\mathcal C}+2\,\frac{\sqrt{\kappa_1\,\kappa_2}}M\)\|\sfPi\,h\nH^2\,.
\ee

Now we use~\eqref{ClaimP1} and~\eqref{ClaimP2} to estimate the weighted $\mathrm H^1$-type quantity
\[
X_2^2:=\irdx{|\nabla_xu_g|^2\,|\nabla_xW_\star|^2\,\rho_\star}\,.
\]
Let us multiply~\eqref{Eqn:g2} by $u_g\,|\nabla_xW_\star|^2\,\rho_\star$ and integrate by parts in order to obtain
\begin{multline*}
\irdx{|u_g|^2\,|\nabla_xW_\star|^2\,\rho_\star}+\irdx{|\nabla_xu_g|^2\,|\nabla_xW_\star|^2\,\rho_\star}\\
+\irdx{(\nabla_xu_g\cdot\nabla_x\psi_g)\,|\nabla_xW_\star|^2\,\rho_\star}+\irdx{u_g\,\nabla_x\(|\nabla_xW_\star|^2\)(\nabla_xu_g+\nabla_x\psi_g)\,\rho_\star}\\
=\irdx{u_h\,u_g\,|\nabla_xW_\star|^2\,\rho_\star}\,.
\end{multline*}
Using Lemma~\ref{Lem:PoincareLambda2}, we obtain
\[
\left|\int_{|x|>R}u_g\,\nabla_x\(|\nabla_xW_\star|^2\)\nabla_xu_g\,\rho_\star\,dx\right|\le\Lambda_\circ\irdx{|u_g|\,|\nabla_xW_\star|^2\,|\nabla_xu_g|\,\rho_\star}\le\Lambda_\circ\,X_1\,X_2\,.
\]
Using~\eqref{Claim1},~\eqref{Claim2}, Lemma~\ref{Lem:PoincareLambda3} and the fact that $1/\rho_\star$ is bounded on $B_R$, we obtain
\[
\left|\int_{|x|\le R}u_g\,\nabla_x\(|\nabla_xW_\star|^2\)\nabla_xu_g\,\rho_\star\,dx\right|\le\mathcal K\,\left\|\nabla_x\(|\nabla_xW_\star\)|^2\right\|_{\mathrm L^\infty(B_R)}\,\|\sfPi\,h\nH^2
\]
and, by similar arguments,
\[
\irdx{(\nabla_xu_g\cdot\nabla_x\psi_g)\,|\nabla_xW_\star|^2\,\rho_\star}\le\kappa_3\,X_2\,\|\sfPi\,h\nH\,,
\]
\[
\irdx{u_g\,\nabla_x\(|\nabla_xW_\star|^2\)\nabla_x\psi_g\,\rho_\star}\le\kappa_4\,X_1\,\|\sfPi\,h\nH\,,
\]
with
\[
\kappa_3:=\left\|\,|\nabla_xW_\star|^2\,\rho_\star\right\|_{\mathrm L^\infty(\R^d,dx)}^{1/2}\quad\mbox{and}\quad\kappa_4:=\left\|\,\big|\nabla_x\(|\nabla_xW_\star|^2\)\big|^2\rho_\star\right\|_{\mathrm L^\infty(\R^d,dx)}^{1/2}\,,
\]
because we know from~\eqref{Claim1} that $\irdx{|\nabla_x\psi_g|^2}\le\|\sfPi\,h\nH^2$. Using Corollary~\ref{Cor:Poincare2}, we obtain that
\[
\(\irdx{u_h\,u_g\,|\nabla_xW_\star|^2\,\rho_\star}\)^2\le\irdx{|u_h|^2\,\rho_\star}\irdx{|u_g|^2\,|\nabla_xW_\star|^4\,\rho_\star}\le\|\sfPi\,h\nH^2\,\frac{X_2^2}{\mathcal C_\circ}\,.
\]
Summarizing, we have shown that
\[
X_1^2+X_2^2-\kappa_3\,X_2\,\|\sfPi\,h\nH-\Lambda_\circ\,X_1\,X_2-\kappa_4\,X_1\,\|\sfPi\,h\nH\le X_2\,\frac{\|\sfPi\,h\nH}{\sqrt{\mathcal C_\circ}}+\mathcal K\,\left\|\nabla_x\(|\nabla_xW_\star\)|^2\right\|_{\mathrm L^\infty(B_R)}\,\|\sfPi\,h\nH^2\,.
\]
Since $X_1^2$ is bounded by $\|\sfPi h\nH^2$, we conclude that
\be{ClaimX}
X_2^2\le\kappa\,\|\sfPi h\nH^2
\ee
for some $\kappa>0$, which has an explicit form in terms quantities involving $\rho_\star$ and its derivatives, as well as all constants in the inequalities of Sections~\ref{Sec:Poincare} and~\ref{Sec:Pointwise}.

\medskip\noindent\emph{Step 4: Second order estimates.} After multiplying~\eqref{Eqn:g2} by $\nabla_x\cdot\big(\rho_\star\,\nabla_xw_g\big)$, we have
\begin{align*}
\irdx{\frac1{\rho_\star}\left|\nabla_x\cdot\(\rho_\star\,\nabla_xw_g\)\right|^2}&=\irdx{(u_h-u_g)\,\nabla_x\cdot\big(\rho_\star\,\nabla_xw_g\big)}\\
&=\irdx{u_h\,\sqrt{\rho_\star}\,\frac1{\sqrt{\rho_\star}}\,\nabla_x\cdot\big(\rho_\star\,\nabla_xw_g\big)}+\irdx{\nabla_xu_g\cdot\nabla_xw_g\,\rho_\star}\\
&\le\frac12\irdx{\(|u_h|^2\,\rho_\star+\frac1{\rho_\star}\left|\nabla_x\cdot\(\rho_\star\,\nabla_xw_g\)\right|^2\)}\\
&\hspace*{5cm}+\frac12\irdx{\(|\nabla_xu_g|^2+|\nabla_xw_g|^2\)\rho_\star}
\end{align*}
and after using~\eqref{Claim1} and~\eqref{Claim2}, we obtain that
\be{Claim3}
\irdx{\frac1{\rho_\star}\left|\nabla_x\cdot\(\rho_\star\,\nabla_xw_g\)\right|^2}\le\(\mathcal K+\frac32\)\|\sfPi\,h\nH^2\,.
\ee

Let $X_3:=\(\irdx{\big(\nabla_xw_g\cdot\nabla_xW_\star\big)^2\,\rho_\star}\)^{1/2}$. After multiplying~\eqref{Eqn:g2} by $\big(\nabla_xw_g\cdot\nabla_xW_\star\big)\,\rho_\star$, we have that
\[
X_3^2-\irdx{\Delta_xw_g\,\big(\nabla_xw_g\cdot\nabla_xW_\star\big)\,\rho_\star}=\irdx{(u_h-u_g)\,\big(\nabla_xw_g\cdot\nabla_xW_\star\big)\,\rho_\star}\,.
\]
Using the Cauchy-Schwarz inequality, we know that the right-hand side can be estimated by $X_3\(\irdx{|u_g|^2\,\rho_\star}\)^{1/2}+X_3\(\irdx{|u_h|^2\,\rho_\star}\)^{1/2}\le2\,X_3\,\|\sfPi\,h\nH$ according to~\eqref{Claim1} and obtain that
\[
X_3^2-2\,X_3\,\|\sfPi\,h\nH\le\irdx{\Delta_xw_g\,\big(\nabla_xw_g\cdot\nabla_xW_\star\big)\,\rho_\star}\,.
\]
Let us notice that
\begin{multline*}
\irdx{\Delta_xw_g\,\big(\nabla_xw_g\cdot\nabla_xW_\star\big)\,\rho_\star}=-\irdx{\Delta_xw_g\,\nabla_xw_g\cdot\nabla_x\rho_\star}\\
=\irdx{\(\mathrm{Hess}(\rho_\star)-\frac 12\,\Delta_x\rho_\star\,\Id\):\nabla_xw_g\otimes\nabla_xw_g}\,.
\end{multline*}
As a consequence, by Lemma~\ref{Lem:PoincareLambda} and~\eqref{Claim1}, we arrive at
\[
X_3^2-2\,X_3\,\|\sfPi\,h\nH\le\frac{\Lambda_\star}2\,\irdx{|\nabla_xw_g|^2\,|\nabla_xW_\star|^2\,\rho_\star}\le{\Lambda_\star}\,X_2^2
+{\Lambda_\star}\,\irdx{|\nabla_x\psi_g|^2\,|\nabla_xW_\star|^2\,\rho_\star}\]
where $X_2^2$ is the quantity that has been estimated in Step 4. Altogether, after taking~\eqref{Claim1} and~\eqref{ClaimX} into account and with
\[
\lambda=\Lambda_\star\(\kappa+\left\|\nabla_xW_\star|^2\,\rho_\star\right\|_{\mathrm L^\infty(\R^d)}\)\,,
\]
which is finite by Lemma~\ref{Lem:PoincareLambda3}, this proves that
\be{Claim4}
\irdx{\big(\nabla_xw_g\cdot\nabla_xW_\star\big)^2\,\rho_\star}\le\(\sqrt{1+\lambda}-1\)^2\,\|\sfPi\,h\nH^2\,.
\ee

\medskip\noindent\emph{Step 5: Conclusion of the proof.} We read from Lemma~\ref{Lem:BLW},~\eqref{Claim0} and~\eqref{Claim3}-\eqref{Claim4} that
\[
\big\|\big(\sfA\sfT(\Id-\sfPi)\big)^*h\nHH\le2\ird{|\mathrm{Hess}(w_g)|^2\rho_\star}\le2\(6\(\mathcal K+\tfrac32\)+8\,\(\sqrt{1+\lambda}-1\)^2\)\,\|\sfPi\,h\nH^2\,,
\]
which concludes the proof of Lemma~\ref{T(1moinsPi)}.\end{proof}

\subsection{Proof of Theorem~\texorpdfstring{\ref{Thm:Main}}{1}}

The potential $V(x)=|x|^\alpha$ satisfies the assumptions~\eqref{Hyp-V1}, \eqref{Hyp-V2}, \eqref{Hyp-V3b}, \eqref{Hyp-V4}, \eqref{Hyp-V5}, \eqref{Hyp-V6}, \eqref{Hyp-V7} and \eqref{Hyp-V8} if $\alpha>1$. The result is then a consequence of Proposition~\ref{DMS} and Lemmas~\ref{Lem:Operators}-\ref{T(1moinsPi)}. A slightly more general result goes as follows.
\begin{theorem}\label{Thm:MainBis} Let us assume that $d\ge1$ and $M>0$. If $V$ satisfies the assumptions~\eqref{Hyp-V1}, \eqref{Hyp-V2}, \eqref{Hyp-V3b}, \eqref{Hyp-V4}, \eqref{Hyp-V5}, \eqref{Hyp-V6}, \eqref{Hyp-V7} and \eqref{Hyp-V8}, then there exist two constants $\lambda>0$ and $\mathcal C>1$ such that any solution $h$ of~\eqref{Syst:VPFPlin} with an initial datum~$h_0$ of zero average such that $\nrm{h_0}{}^2<\infty$ satisfies
\[
\nrm{h(t,\cdot,\cdot)}{}^2\le\mathcal C\,\nrm{h_0}{}^2\,e^{-\lambda t}\quad\forall\,t\ge0\,.
\]
\end{theorem}

\section{Uniform estimates in the diffusion limit}\label{Sec:DiffusionLimit}

The hypocoercivity method of~\cite{Dolbeault2009511,MR3324910} is directly inspired by the drift-diffusion limit, as it relies on a micro/macro decomposition in which the relaxation in the velocity direction is given by the microscopic coercivity property~\eqref{H1} while the relaxation in the position direction arises from the macroscopic coercivity property~\eqref{H2} which governs the relaxation of the solution of the drift-diffusion equation obtained as a limit.

\subsection{Formal macroscopic limit.}

Let us start with a formal analysis in the framework of Section~\ref{Sec:Abstract}, when~\eqref{EqnEvol} is replaced by the scaled evolution equation
\be{EqnEvol-ve}
\ve\,\frac{dF}{dt}+\sfT F=\frac1\ve\,\sfL F
\ee
on the Hilbert space $\mathcal H$. We assume that a solution $F_\ve$ of~\eqref{EqnEvol-ve} can be expanded as
\[
F_\ve=F_0+\ve\,F_1+\ve^2\,F_2+{\mathcal{O}(\ve^3)}
\]
in the asymptotic regime corresponding to $\ve\to0_+$ and, at formal level, that~\eqref{EqnEvol-ve} can be solved order by order:
\[
\begin{array}{ll}
\ve^{-1}:&\quad\sfL F_0=0\,,\\[4pt]
\ve^0:&\quad\sfT F_0=\sfL F_1\,,\\[4pt]
\ve^1:&\quad\frac{dF_0}{dt}+\sfT F_1=\sfL F_2\,.
\end{array}
\]
The first equation reads as $F_0=\sfPi F_0$, that is, $F_0$ is in the kernel of $\sfL$. Assume for simplicity that $\sfL^{-1}\,(\sfT\sfPi)=-\,\sfT\sfPi$ on an appropriate subspace, so that the second equation is simply solved by $F_1=-\,(\sfT\sfPi)\,F_0$. Let us consider the projection on the kernel of the $\mathcal O(\ve^1)$ equation:
\[
\frac d{dt}\(\sfPi F_0\)-\,\sfPi\sfT\,(\sfT\sfPi)\,F_0=\sfPi\sfL F_2=0\,.
\]
If we denote by $u$ the quantity $F_0=\sfPi F_0$ and use~\eqref{H3}, then $-\,(\sfPi\sfT)\,(\sfT\sfPi)=(\sfT\sfPi)^*\,(\sfT\sfPi)$ and the equation becomes
\[\label{DDlimit}
\partial_tu+(\sfT\sfPi)^*\,(\sfT\sfPi)\,u=0\,,
\]
which is our \emph{drift-diffusion} limit equation. Notice that if $u$ solves this equation, then
\[
\frac d{dt}\|u\nH^2=-\,2\,\|(\sfT\sfPi)\,u\nH^2\le-\,2\,\lambda_M\,\|u\nH^2
\]
according to~\eqref{H2}. This program applies in the case of the scaled evolution equation~\eqref{Syst:VPFPlinParab}. Let us give a few additional details.

\medskip Let us assume that a solution $h_\ve$ of~\eqref{Syst:VPFPlinParab} can be expanded as $h_\ve=h_0+\ve\,h_1+\ve^2\,h_2+{\mathcal{O}(\ve^3)}$, in the asymptotic regime as $\ve\to0_+$. Solving~\eqref{Syst:VPFPlinParab} order by order in $\ve$, we find the equations
\[
\begin{array}{ll}
\ve^{-1}:&\quad\Delta_vh_0-v\cdot\nabla_vh_0=0\,,\\[4pt]
\ve^0:&\quad v\cdot\nabla_xh_0-\nabla_xW_\star\cdot\nabla_vh_0+v\cdot\nabla_x\psi_{h_0}=\Delta_vh_1-v\cdot\nabla_vh_1\,,\\[4pt]
\ve^1:\quad&\quad\partial_th_0+v\cdot\nabla_xh_1-\nabla_xW_\star\cdot\nabla_vh_1+v\cdot\nabla_x\psi_{h_1}=\Delta_vh_2-v\cdot\nabla_vh_2\,.
\end{array}
\]
Let us define $u=\sfPi h_0$, $\psi=\psi_{h_0}$ such that $-\Delta_x\psi=u\,\rho_\star$, $w=u+\psi$ and observe that the first two equations simply mean
\[
u=h_0\,,\quad v\cdot\nabla_xw=\Delta_vh_1-v\cdot\nabla_vh_1\,,
\]
from which we deduce that $h_1=-\,v\cdot\nabla_xw$. After projecting with~$\sfPi$, the third equation is 
\[
\partial_tu-\Delta_xw+\nabla_xW_\star\cdot\nabla_xw=0\,,
\]
using $\irdv{v\otimes v\,\mathcal M(v)}=\Id$. If we define $\rho=u\,\rho_\star$, we have formally obtained that it solves
\[
\partial_t\rho=\Delta_x\rho+\nabla_x\cdot\Big(\rho\(\nabla_xV+\nabla_x\phi_\star\)\Big)+\nabla_x\cdot\(\rho_\star\,\nabla_x\psi\)\,,\quad-\,\Delta_x\psi=\rho\,.
\]
At this point, we can notice that the solution $\rho$ converges to $\rho_\star$ according to the results of, \emph{e.g.},~\cite{Nernst-Planck}, at an exponential rate which is independent of $\ve$.

\subsection{Hypocoercivity}

Let us adapt the computations of Section~\ref{Sec:Abstract} to the case $\ve\in(0,1)$ as in~\cite{BDMMS}. If $F$ solves~\eqref{EqnEvol-ve}, then
\begin{multline*}
-\,\ve\,\frac d{dt}\sfH[F]=\mathsf D_{\delta,\ve}[F]\,,\\
\mathsf D_{\delta,\ve}[F]:=-\,\frac1\ve\,\scalar{\sfL F}F+\delta\,\scalar{\sfA\sfT\sfPi F}F-\,\delta\,\scalar{\sfT\sfA F}F+\delta\,\scalar{\sfA\sfT(\Id-\sfPi)F}F-\frac\delta\ve\,\scalar{\sfA\sfL F}F\,.
\end{multline*}
The estimates are therefore exactly the same as in Proposition~\ref{DMS}, up to the replacement of~$\lambda_m$ by $\lambda_m/\ve$ and $C_M$ by $C_M/\ve$. Hence, for $\ve>0$ small enough, we have that
\[
\delta(\ve):=\min\left\{2,\,\frac{\lambda_m}\ve,\,\ve\,\lambda_\star(\ve)\right\}=\frac{4\,\lambda_m\,\lambda_M\,\ve}{4\,\lambda_M\,\ve^2+C_M^2\,(1+\lambda_M)}\,.
\]
We may notice that $\lim_{\ve\to0_+}\frac{\delta(\ve)}\ve=2\,\zeta$ with
\[
\zeta:=\frac{2\,\lambda_m\,\lambda_M}{C_M^2\,(1+\lambda_M)}
\]
and, for $\ve>0$ small enough,
\[
\frac{2-\,\zeta\,\ve}4\,\|F\nH^2\le\mathsf H_{\,\zeta\,\ve}[F]\le\frac{2+\zeta\,\ve}4\,\|F\nH^2\quad\forall F\in\mathcal H\,.
\]
By revisiting the proof of Proposition~\ref{DMS}, we find that with $\delta=\zeta\,\ve$ and $\lambda=\eta\,\ve$ with
\[
\eta:=\frac{\lambda_m\,\lambda_M^2}{C_M^2\,(1+\lambda_M)^2}\,,
\]
the quadratic form
\[
(X,Y)\mapsto\(\frac{\lambda_m}\ve-\,\delta\)X^2+\frac{\delta\,\lambda_M}{1+\lambda_M}\,Y^2-\,\delta\,\frac{C_M}\ve\,X\,Y-\frac\lambda2\,\(X^2+Y^2\)-\frac\lambda2\,\delta\,X\,Y
\]
is a nonnegative quadratic form for $\ve>0$ small enough. In the regime as $\ve\to0_+$, the result of Proposition~\ref{DMS} can be adapted as follows.
\begin{corollary}\label{DMS2} Assume that~\eqref{H1}--\eqref{H4} hold and take $\zeta$ as above. Then for $\ve>0$ small enough,
\[
\eta\,\ve\,\mathsf H_{\,\zeta\,\ve}[F]\le\mathsf D_{\,\zeta\,\ve,\ve}[F]\quad\forall F\in\mathcal H\,.
\]
\end{corollary}
\begin{proof} The range for which the quadratic form is negative is given by the condition
\[
\lambda_m^2\,K^4\,\ve^4+K\,C_M^3\,\big(4\,K\,\lambda_m+3\,C_M\,(K+4)\big)\,\ve^2-2\,C_M^6<0\,.
\]
It follows that the above condition is satisfied if $\ve$ is taken small enough which, for the same reasons as above in this paper, guarantees that the entropy-entropy production inequality of Corollary~\ref{DMS2} holds.\end{proof}

As an easy consequence, if $F_\ve$ solves~\eqref{EqnEvol-ve}, we have that
\[
\mathsf H_{\,\zeta\,\ve}[F(t,\cdot)]\le\mathsf H_{\,\zeta\,\ve}[F(0,\cdot)]\,e^{-\eta\,t}\quad\forall\,t\ge0\,.
\]

\begin{proof}[Proof of Theorem~\ref{Thm:DiffusionLimit}] With the abstract result on~\eqref{EqnEvol-ve} applied to~\eqref{Syst:VPFPlinParab}, the estimate~\eqref{HypocoRate} holds with $\delta=\zeta\,\ve$. We conclude with $\lambda=\eta$, for some $\mathcal C>1$, which do not depend on $\ve\to0_+$.\end{proof}

\section{The nonlinear system in dimension \texorpdfstring{$d=1$}{d=1}}\label{Sec:d=1}

\noindent With the notation~\eqref{TL}, we can rewrite the Vlasov-Poisson-Fokker-Planck system~\eqref{Syst:VPFP}~as
\[
\partial_th+\sfT h=\sfL h+\mathsf Q[h]\,,\quad-\Delta_x\psi_h=\irdv{h\,f_\star}\,,\quad\mbox{with}\quad\mathsf Q[h]:=\nabla_x\psi_h\cdot\(\nabla_vh-v\,h\)\,.
\]
Here we assume that $d=1$ and prove Corollary~\ref{Cor:d=1}. Using the representation~\eqref{Rep1d}, so that
\[
\psi_h'(x)=-\int_{-\infty}^xu_h\,\rho_\star\,dx\quad\forall\,x\in\R\,,
\]
and the convergence of $h(t,\cdot,\cdot)\to0$ in $\mathrm L^1(\R\times\R,d\mu)$ as $t\to+\infty$, as a consequence of~\cite{MR1306570}, we learn that $t\mapsto\|\psi_h'(t,\cdot)\|_{\mathrm L^\infty(\R)}$ is bounded uniformly w.r.t.~$t\ge0$. In fact, we have a slightly more precise estimate that goes as follows.
\begin{lemma}\label{Lem:CK} Assume $V$ satisfies~\eqref{Hyp-V1} and~\eqref{Hyp-V2} and let $\rho_\star\in\mathrm L^1(\R^d)$ be the solution of~\eqref{Eqn:Poisson-Boltzmann} such that $\irdx{\rho_\star}=M$. Let $f=(1+h)\,f_\star\in\mathrm L^1_+(\R^d\times\R^d)$ such that $\iint_{\R\times\R}f\,\log(f/f_\star)\,dx\,dv<\infty$. Under the assumption $\iint_{\R\times\R}h\,f_\star\,dx\,dv=0$, $\psi_h'$ as defined above satisfies the estimate
\[
\|\psi_h'\|_{\mathrm L^\infty(\R)}^2\le4\,M\,\iint_{\R\times\R}f\,\log\(\frac f{f_\star}\)\,dx\,dv\,.
\]
Additionally, under the assumptions of Corollary~\ref{Cor:d=1}, if $h$ solves~\eqref{Syst:VPFP}, then
\[
\lim_{t\to+\infty}\|\psi_h'(t,\cdot)\|_{\mathrm L^\infty(\R)}=0\,.
\]
\end{lemma}
\begin{proof} We deduce from Jensen's inequality
\[
\int_\R f\,\log\(\frac f{\mathcal M}\)\,dv\ge\rho_h\,\log\rho_h
\]
that $f/f_\star=1+h$ is such that the \emph{free energy} satisfies the bound
\[
\iint_{\R\times\R}(1+h)\,\log(1+h)\,f_\star\,dx\,dv\ge\int_\R\rho_h\,\log\(\frac{\rho_h}{\rho_\star}\)dx=\int_\R(1+u_h)\,\log(1+u_h)\,\rho_\star\,dx
\]
and get according to~\cite{Csiszar67,Kullback67,MR0213190} from the \emph{Csisz\'ar-Kullback-Pinsker inequality} that
\[
\int_\R(1+u_h)\,\log(1+u_h)\,\rho_\star\,dx\ge\frac1{4\,M}\(\int_\R|u_h|\,\rho_\star\,dx\)^2\ge\frac{\|\psi_h'\|_{\mathrm L^\infty(\R)}^2}{4\,M}\,.
\]

Concerning the evolution problem~\eqref{Syst:VPFP}, we recall that the free energy decays according to
\[
\frac d{dt}\(\iint_{\R\times\R}(1+h)\,\log(1+h)\,f_\star\,dx\,dv+\frac12\int_\R|\psi_h'|^2\,dx\)=-\iint_{\R\times\R}{f\,\left|\nabla_v\log\(\frac f{f_\star}\)\right|^2\,dx\,dv}\,,
\]
where the right-hand side is, up to the sign, a Fisher information. As stated in~\cite[Theorem~B]{MR1306570}, this shows the strong convergence of $f(t+n,\cdot,\cdot)$ to $f_\star$ in $\mathrm L^1(\R^+\times\R^d\times\R^d)$ as $n\to+\infty$ because~$f_\star$ is the unique solution of~\eqref{Syst:VPFP} with mass $M$ and a Maxwellian velocity distribution. By the logarithmic Sobolev inequality, this also proves that the limit of the free energy is $0$, which concludes the proof of Lemma~\ref{Lem:CK}.\end{proof}

\begin{proof}[Proof of Corollary~\ref{Cor:d=1}.] With the notations of Section~\ref{Sec:ScalarProduct} and the functional $\sfH$ defined as in the linear case by
\[
\sfH[h]:=\tfrac12\,\|h\nH^2+\delta\,\scalar{\sfA h}h\,,
\]
we obtain that
\begin{multline*}
\frac d{dt}\sfH[h]+\scalar{\sfL h}h-\,\delta\,\scalar{\sfA\sfT\sfPi h}h+\,\delta\,\scalar{\sfT\sfA h}h-\,\delta\,\scalar{\sfA\sfT(\Id-\sfPi)h}h+\,\delta\,\scalar{\sfA\sfL h}h\\
=\scalar{\mathsf Q[h]}h+\delta\,\scalar{\sfA\mathsf Q[h]}h+\delta\,\scalar{\mathsf Q[h]}{\sfA h}\,.
\end{multline*}
Let us give an estimate of the three terms of the right hand side.

\smallskip\noindent 1) In order to estimate
\[
\scalar{\mathsf Q[h]}h=\iint_{\R\times\R}\psi_h'\,(\partial_vh-v\,h)\,h\,f_\star\,dx\,dv+\int_\R\psi_h'\,\rho_\star\(\int_\R(\partial_vh-v\,h)\,\mathcal M\,dv\)\,\psi_h\,dx\,,
\]
we notice that $\int_\R(\partial_vh-v\,h)\,\mathcal M\,dv=0$ and also that $\iint_{\R\times\R}|\partial_v h|^2\,f_\star\,dx\,dv=-\scalar{\sfL h}h$. From the improved Poincar\'e inequality~\cite[Ineq.~(4)]{MR2956120}, we also learn that $\|v\,h\nH^2\le2\,(d+2)\,\|\nabla_vh\nH^2$. Simple Cauchy-Schwarz inequalities show that
\[
\big|\scalar{\mathsf Q[h]}h\big|\le c\,\|\psi_h'\|_{\mathrm L^\infty(\R)}\,\big|\scalar{\sfL h}h\big|^{1/2}\,\|\sfPi h\nH
\]
because $\nrm h{\mathrm L^2(\R^d\times\R^d,d\mu)}\le\|h\nH$, with $c=1+\sqrt{2\,(d+2)}$.

\smallskip\noindent 2) Let us consider $g=\sfA h=u_g$ given by
\[
u_g-\frac1{\rho_\star}\,\nabla_x\cdot\big(\rho_\star\,\nabla_xw_g\big)=-\frac1{\rho_\star}\,\nabla_x\cdot j_h\quad\mbox{with}\quad j_h:=\irdv{v\,h\,f_\star}\,.
\]
With $\psi_g$ such that $-\psi_g''=u_g\,\rho_\star$, we have to estimate
\[
\scalar{\mathsf Q[h]}{\sfA h}=\iint_{\R\times\R}\psi_h'\,(\partial_vh-v\,h)\,u_g\,f_\star\,dx\,dv+\int_\R\psi_h'\,\rho_\star\(\int_\R(\partial_vh-v\,h)\,\mathcal M\,dv\)\,\psi_g\,dx\,.
\]
Exactly as above, we have on the one hand that
\begin{multline*}
\left|\iint_{\R\times\R}\psi_h'\,(\partial_vh-v\,h)\,u_g\,f_\star\,dx\,dv\right|\le\|\psi_h'\|_{\mathrm L^\infty(\R)}\,\|g\nH\,\|\partial_vh-v\,h\nH\\
\le c\,\|\psi_h'\|_{\mathrm L^\infty(\R)}\,\|(\Id-\sfPi)h\nH\,\big|\scalar{\sfL h}h\big|^{1/2}
\end{multline*}
because $\nrm g{\mathrm L^2(\R^d\times\R^d,d\mu)}\le\|g\nH=\|\sfA h\nH\le\|(\Id-\sfPi)h\nH$, and on the other hand that
\begin{multline*}
\irdxi{|\psi_g|^2\,\rho_\star}\le\mathcal C_\star^{-1}\irdxi{|\psi_g'|^2\,\rho_\star}+\(\irdxi{\psi_g\,\rho_\star}\)^2\\
\le\frac{\|\rho_\star\|_{\mathrm L^\infty(\R)}}{\mathcal C_\star}\irdxi{|\psi_g'|^2}+\irdxi{|u_g|^2\,\rho_\star}\irdxi{|\phi_\star|^2\,\rho_\star}
\end{multline*}
by Lemma~\ref{Lem:Poincare} again, from which we conclude that
\[
\big|\scalar{\mathsf Q[h]}{\sfA h}\big|\le c\,\|\psi_h'\|_{\mathrm L^\infty(\R)}\,\big|\scalar{\sfL h}h\big|^{1/2}\,\|(\Id-\sfPi)h\nH\,.
\]

\smallskip\noindent 3) With $g$ given in terms of $h$ by~\eqref{Eqn:g}, $\sfA^*h=v\,w_g'$ and we learn from~\eqref{Claim1} that $\|\sfA^*h\nH\le\|\sfPi h\nH$. Hence 
\[
\big|\scalar{\sfA\mathsf Q[h]}h\big|=\big|\scalar{\mathsf Q[h]}{\sfA^* h}\big|\le c\,\|\psi_h'\|_{\mathrm L^\infty(\R)}\,\big|\scalar{\sfL h}h\big|^{1/2}\,\|\sfPi h\nH\,.
\]

\medskip Summing up all these estimates and using $-\,\scalar{\sfL h}h\ge\lambda_m\,\|(\Id-\sfPi)h\nH^2$ by Lemma~\ref{Lem:MicroscopicCoercivity}, we obtain as in the proof of Proposition~\ref{DMS} that
\[
\frac d{dt}\sfH[h]\le-\,\lambda\,\sfH[h]
\]
for the largest value of $\lambda$ for which
\[
(X,Y)\mapsto(\lambda_m-\,\delta)\,X^2+\frac{\delta\,\lambda_M}{1+\lambda_M}\,Y^2-\,\delta\,C_M\,X\,Y-\frac\lambda2\,\(X^2+Y^2\)-\frac\lambda2\,\delta\,X\,Y-\epsilon\,X\,(X+2\,Y)
\]
is a nonnegative quadratic form, as a function of $(X,Y)$. Here $X:=\|(\Id-\sfPi)h\nH$, $Y:=\|\sfPi h\nH$, and
\[
\epsilon:=c\,\|\psi_h'\|_{\mathrm L^\infty(\R)}
\]
can be taken as small as we wish, if we assume that $t>0$ is large enough. This completes the proof of Corollary~\ref{Cor:d=1}.\end{proof}

Let us conclude this section by some remarks.\begin{itemize}
\item[(i)] It is clear from the proof of Corollary~\ref{Cor:d=1} that the optimal rate is as close as desired of the optimal rate in the linearized problem~\eqref{Syst:VPFPlin} obtained in Theorem~\ref{Thm:Main}. Up to a change of the constant $\mathcal C$, we can actually establish that these rates are equal because we read form the above proof that $\epsilon(t)=\mathcal O\(e^{-\lambda t}\)$ and the result follows from a simple ODE argument. This is a standard observation in entropy methods, which has been used on many occasions: see for instance~\cite{MR2481073}.
\item[(ii)] Corollary~\ref{Cor:d=1} is written for $V(x)=|x|^\alpha$ but it is clear that it can be extended to the setting of Theorem~\ref{Thm:MainBis}. Similarly, our estimates are compatible with the diffusion limit, as in Section~\ref{Sec:DiffusionLimit}.
\item[(iii)] Results in higher dimensions, \emph{i.e.}, for $d\ge2$ as in~\cite{Ju_Hwang_2013,MR3522010,MR3767000,Bedrossian_2017} rely on smallness conditions, special properties of the potential $V$ (typically, $V\equiv0$ or $V(x)=|x|^2$), or closure conditions on regularity estimates which do not allow to handle the decay of generic solutions of~\eqref{Syst:VPFP} based on the properties of the free energy, as we do above in the case $d=1$. This is so far an important open question, which deserves attention. The understanding of the mechanism should go through a detailed description of the smoothing and decay properties of the solutions for large time asymptotics.
\end{itemize}


\bigskip\noindent{\bf Acknowledgments:} The Authors would like to thank two referees for very helpful comments and suggestions. This work has been partially supported by the Project EFI (ANR-17-CE40-0030) of the French National Research Agency (ANR) and by the research unit \emph{Dynamical systems and their applications} (UR17ES21), Ministry of Higher Education and Scientific Research, Tunisia. 
\\\noindent{\scriptsize\copyright\,2021 by the authors. This paper may be reproduced, in its entirety, for non-commercial purposes.}


\begin{center}
\rule{2cm}{0.5pt}
\end{center}

\parindent=0pt\parskip=4pt
L.~Addala: Department of Mathematics, Faculty of Sciences of Bizerte,\\ University of Carthage, 7021 Zarzouna, Banzart, Tunisia.\\\Email{lanoiraddala@gmail.com}

J.~Dolbeault: CEREMADE (CNRS UMR n$^\circ$ 7534), PSL research university,\\ Universit\'e Paris-Dauphine, Place de Lattre de Tassigny, 75775 Paris 16, France.\\ \Email{dolbeaul@ceremade.dauphine.fr}

X.~Li: CEREMADE (CNRS UMR n$^\circ$ 7534), PSL research university,\\ Universit\'e Paris-Dauphine, Place de Lattre de Tassigny, 75775 Paris 16, France.\\ \Email{li@ceremade.dauphine.fr}

M.L.~Tayeb: Department of Mathematics, Faculty of Sciences of Tunis,\\ University of Tunis El Manar, 2092 El-Manar, Tunisia.\\\Email{lazhar.tayeb@fst.rnu.tn}

\end{document}